\documentclass[a4paper,12pt]{amsart}
\usepackage[utf8]{inputenc}

\usepackage{url}
\usepackage[margin=1in]{geometry}  % set the margins to 1in on all sides
\usepackage{epsfig}
\usepackage{color}
\usepackage{graphicx}              % to include figures
\usepackage{amsmath}
\usepackage{amscd}               % great math stuff
\usepackage{amsfonts}
\usepackage{amssymb}             % for blackboard bold, etc
\usepackage{amsthm}                % better theorem environments
\usepackage{epsfig}
\usepackage{mathrsfs}
\usepackage{hyperref}
\usepackage{enumerate}
\usepackage{esint}

\newtheorem{theorem}{Theorem}

\newtheorem*{theorem*}{Theorem}
\newtheorem{lemma}[theorem]{Lemma}
\newtheorem*{lemma*}{Lemma}
\newtheorem{proposition}[theorem]{Proposition}
\newtheorem*{proposition*}{Proposition}
\theoremstyle{definition}

\theoremstyle{remark}

\newcommand{\tr}{\operatorname{tr}}
\newcommand{\Disc}{\operatorname{Disc}}
\newcommand{\RE}{\operatorname{Re}}
\newcommand{\IM}{\operatorname{Im}}

\newcommand{\sgn}{\operatorname{sgn}}

\newcommand{\sym}{\operatorname{sym}}

\newcommand{\E}{\mathbf{E}}     %mathematical expectation
\newcommand{\one}{\mathbf{1}}

\newcommand{\Stab}{\operatorname{Stab}}

\newcommand{\bC}{\mathbb{C}}

\newcommand{\bR}{\mathbb{R}}

\newcommand{\zed}{\mathbb{Z}}

\newcommand{\GL}{\mathrm{GL}}

\newcommand{\SL}{\mathrm{SL}}
\newcommand{\SO}{\mathrm{SO}}

\newcommand{\sE}{{\mathscr{E}}}

\newcommand{\sI}{{\mathscr{I}}}

\newcommand{\sL}{{\mathscr{L}}}

\newcommand{\sR}{{\mathscr{R}}}
\newcommand{\sS}{{\mathscr{S}}}

\newcommand{\fS}{\mathfrak{S}}

\title{Eisenstein series twisted Shintani zeta function}
\author{Robert D. Hough}
\address{Department of Mathematics, Stony Brook University, 100 Nicolls Road, Stony Brook, NY 11794}
\email{robert.hough@stonybrook.edu}
\author{Eun Hye Lee}
\address{Department of Mathematics, Stony Brook University, 100 Nicolls Road, Stony Brook, NY 11794}
\email{eunhye.lee@stonybrook.edu }
\begin{document}

\begin{abstract}
We introduce the zeta function of the prehomogenous vector space of binary cubic forms, twisted by the real analytic Eisenstein series.  We prove the meromorphic continuation of this zeta function and identify its poles and their residues.  We also identify the poles and residues of the zeta function when restricted to irreducible binary cubic forms. This zeta function can be used to prove the equidistribution of the lattice shape of cubic rings.
\end{abstract}

\thanks{This material is based upon work supported by the National Science
Foundation under agreement DMS-1802336. Any opinions, findings and
conclusions or recommendations expressed in this material are those of the
authors and do not necessarily reflect the views of the National Science
Foundation.}

\thanks{Robert Hough is supported by an Alfred P. Sloan Foundation Research 
Fellowship and a Stony Brook Trustees Faculty Award}

\maketitle

\section{Introduction}
The study of integral orbits ordered by invariants in a representation space is a major are of current development in number theory, with applications to arithmetic statistics, see \cite{B04a}, \cite{B04b}, \cite{BG14}, \cite{BH16}, \cite{BST13}, \cite{BS15a}, \cite{BS15b}, \cite{BV15}, \cite{B20}.  The area is being developed in several directions, in terms of the representation spaces, the orbit description of the local conditions and their Fourier transform \cite{TT20a}, \cite{TT20b}, \cite{H20}, and in terms of the zeta function enumerating the orbits \cite{S72}, \cite{SS74}, \cite{WY92}, \cite{Y93}, \cite{Y97}.  The purpose of this paper is to study the distribution of the orbits with respect to the spectral expansion of the underlying homogeneous space, which gives a method of proving rates in quantitative equidistribution statements such as those in \cite{BH16} and \cite{T97}, see \cite{H19} where a rate is obtained in the cuspidal spectrum, similar to Duke's theorem \cite{D88}, \cite{KS93}. 

A cubic ring is a free rank three $\zed$ module with a ring structure.  Let a basis be $\langle 1, \omega, \theta\rangle$.  There is a natural action of $\GL_2$ which forms linear combinations of $\omega$ and $\theta$ modulo 1.  After tensoring with $\bR$, a cubic ring can be identified with a three dimensional lattice in either $\bR^3$ or $\bR \times \bC$.  Define the lattice shape of the ring to be the lattice shape of this lattice projected in the two dimensional plane orthogonal to 1, and determined up to homothety.  There is a well-known discriminant preserving bijective correspondence between cubic rings up to isomorphism and binary cubic forms \cite{DF64}, \cite{GGS02}.
Shintani \cite{S72} introduced $\zeta$ functions enumerating cubic rings up to isomorphism ordered by discriminant, and determined the poles and residues of these zeta functions and proved a functional equation.  His method was used by Taniguchi and Thorne \cite{TT13a}, \cite{TT13b} to prove a secondary main term in the Davenport-Heilbronn Theorem counting cubic fields.  In \cite{H17} the first author introduced a twisted version of the Shintani zeta function, in which a Maass cusp form is evaluated on the lattice shape, and in \cite{H19} this is used to prove quantitative equidistribution of the lattice shape of the ring of integers of cubic fields in the canonical embedding in the cuspidal part of the spectrum.  This article complements \cite{H17} by determining the poles and residues of the zeta function twisted by a real analytic Eisenstein series.

Let $V_+, V_-$ denote the spaces of real binary cubic forms with positive or negative discriminant.  
Let $x_{+} = \frac{1}{(108)^{\frac{1}{4}}}(3v^2w - w^3)$ and $x_{-} = \frac{1}{\sqrt{2}}(v^2w + w^3)$.  The group $G^+ = \{g \in \GL_2(\bR): \det(g)>0\}$ is a three-fold cover of $V_+$ by $g \mapsto g \cdot x_+$ and a single cover of $V_-$ by $g \mapsto g \cdot x_-$.  The stabilizer of $x_+$ is the rotation group of order 3.  By identifying a point $x \in V_{\pm}$ with $g \in \Gamma \backslash \SL_2(\bR),$ $\Gamma = \SL_2(\zed)$ such that $g \cdot x_{\pm} = x$ up to homothety, there is an identification of cubic rings with the shape of the ring in the space of two dimensional lattices $\Gamma \backslash \SL_2(\bR)$.

For non-zero integer $m$ let $\{x_{i,m}\}_{i=1}^{h(m)}$ be representatives for the classes of integral binary cubic forms of discriminant $m$. Choose group elements $\{g_{i,m}\}_{i=1}^{h(m)}$ group elements so that $g_{i,m}\cdot x_{\sgn(m)} = x_{i,m}$.  Associated to real analytic Eisenstein series $\E_r$, $r = \frac{1 + z^2}{4}$ are the twisted zeta functions
\begin{align*}
 \sL^+(\E_r, s) &= \sum_{m=1}^\infty \frac{1}{m^s} \sum_{i=1}^{h(m)}  \frac{\E_r(g_{i,m})}{|\Stab(x_{i,m})|},  \qquad \RE(s) > \frac{5}{4}\\
 \sL^-(\E_r, s) &= \sum_{m=1}^\infty \frac{1}{m^s} \sum_{i=1}^{h(-m)}  \E_r(g_{i,m}), \qquad \RE(s)> \frac{5}{4}.
\end{align*}
%Although the Eisenstein series are unbounded, their growth is bounded by order $t$ and hence the barrier to equidistribution together with Shintani's original counts of cubic orders guarantees the absolute convergence of these series.  Whereas in the cuspidal case, where the twisted zeta functions are entire, we demonstrate that the Eisenstein twists have two poles, one each corresponding to a barrier to equidistribution.

\begin{theorem}\label{main_theorem}
 The real analytic Eisenstein twisted zeta functions have meromorphic continuation to $\bC$, with poles at $\frac{5\pm z}{4}$ and $\frac{11 \pm z}{12}$ with residues listed in the following table. 
 
 %\begin{tabular}{|l|l|l|}
 %\hline
 %Pole &  $\frac{11 + z}{12}$ & $\frac{11-z}{12}$\\
 % \hline
 % $\sL^-$  &$\frac{\zeta\left(\frac{1-z}{3} \right)2^{\frac{z-1}{6}}\pi^{\frac{2z+1}{6}}}{3}\cos\left(\frac{\pi(1-z)}{6}\right)\frac{\Gamma\left(\frac{1-z}{3}\right)\Gamma\left(\frac{4-z}{6}\right)}{\Gamma\left(\frac{7-z}{6}\right)}$ &$\frac{\xi(z)}{\xi(1+z)}\frac{\zeta\left(\frac{1+z}{3} \right)2^{\frac{-z-1}{6}}\pi^{\frac{-2z+1}{6}}}{3}\cos\left(\frac{\pi(1+z)}{6}\right)\frac{\Gamma\left(\frac{1+z}{3}\right)\Gamma\left(\frac{4+z}{6}\right)}{\Gamma\left(\frac{7+z}{6}\right)}$\\
%  \hline
%  $\sL^+$  &$\frac{\zeta\left(\frac{1-z}{3} \right)2^{\frac{z-1}{6}}\pi^{\frac{2z+1}{6}}}{3^{\frac{7-z}{4}}}\cos\left(\frac{\pi(1-z)}{6}\right)\frac{\Gamma\left(\frac{1-z}{3}\right)\Gamma\left(\frac{4-z}{6}\right)}{\Gamma\left(\frac{7-z}{6}\right)}$ &$\frac{\xi(z)}{\xi(1+z)}\frac{\zeta\left(\frac{1+z}{3} \right)2^{\frac{-z-1}{6}}\pi^{\frac{-2z+1}{6}}}{3^{\frac{7+z}{4}}}\cos\left(\frac{\pi(1+z)}{6}\right)\frac{\Gamma\left(\frac{1+z}{3}\right)\Gamma\left(\frac{4+z}{6}\right)}{\Gamma\left(\frac{7+z}{6}\right)}$\\
%  \hline  
%  \hline
% Pole & $\frac{5 + z}{4}$& $\frac{5-z}{4}$ \\
%\hline
%  $\sL^-$ & $\zeta(3+z)2^{\frac{-5-z}{2}}$&$\frac{\xi(z)}{\xi(1+z)}\zeta(3-z)2^{\frac{-5+z}{2}}$\\
%\hline
%  $\sL^+$ & $\zeta(3+z)2^{\frac{-5-z}{2}}3^{\frac{1+z}{4}}$&$\frac{\xi(z)}{\xi(1+z)}\zeta(3-z)2^{\frac{-5+z}{2}}3^{\frac{1-z}{4}}$\\
%\hline
%  \end{tabular}{|l|l|l|}

\begin{tabular}{|l|l|l|}
\hline
 Pole & $\frac{11 + z}{12}$ & $\frac{5 + z}{4}$ \\
 \hline
 $\sL^-$ & $\frac{\zeta\left(\frac{1-z}{3} \right)2^{\frac{z-1}{6}}\pi^{\frac{2z+1}{6}}}{3}\cos\left(\frac{\pi(1-z)}{6}\right)\frac{\Gamma\left(\frac{1-z}{3}\right)\Gamma\left(\frac{4-z}{6}\right)}{\Gamma\left(\frac{7-z}{6}\right)}$& $\zeta(3+z)2^{\frac{-5-z}{2}}$\\ \hline 
 $\sL^+$ & $\frac{\zeta\left(\frac{1-z}{3} \right)2^{\frac{z-1}{6}}\pi^{\frac{2z+1}{6}}}{3^{\frac{7-z}{4}}}\cos\left(\frac{\pi(1-z)}{6}\right)\frac{\Gamma\left(\frac{1-z}{3}\right)\Gamma\left(\frac{4-z}{6}\right)}{\Gamma\left(\frac{7-z}{6}\right)}$& $\zeta(3+z)2^{\frac{-5-z}{2}}3^{\frac{1+z}{4}}$\\
 \hline
\end{tabular}

The poles at $\frac{11 - z}{12}$ and $\frac{5-z}{4}$ are found by replacing $z$ with $-z$ and multiplying by $\frac{\xi(z)}{\xi(1+z)}$.
 
\end{theorem}
We also study the twisted zeta functions in which summation is restricted to irreducible forms.  Let
\begin{align*}
 \sL^{+,i}(\E_r, s) &= \sum_{m=1}^\infty \frac{1}{m^s} \sum_{\substack{i = 1, \\x_{i,m} \text{ irreducible}}}^{h(m)} \frac{\E_r(g_{i,m})}{|\Stab(x_{i,m})|}, \qquad \RE(s)> \frac{5}{4},\\
 \sL^{-,i}(\E_r, s) &= \sum_{m=1}^\infty \frac{1}{m^s} \sum_{\substack{i=1,\\ x_{i,m} \text{ irreducible}}}^{h(-m)}  \E_r(g_{i,m}), \qquad \RE(s)> \frac{5}{4}.
\end{align*}

\begin{theorem}\label{irreducible_theorem}
 The irreducible twisted $\sL$ functions have meromorphic continuation to $\RE(s)> \frac{3}{4}$ with poles at $\frac{11 \pm z}{12}$ with residues equal to those from Theorem \ref{main_theorem}.
\end{theorem}

\subsection*{Discussion}
  In his thesis, Terr \cite{T97} proved that this lattice shape is asymptotically equidistributed with respect to the induced Haar measure when cubic orders are ordered by growing size of discriminant.  After Terr's work, it was noticed that there is an evident obstruction to uniformity.  When $g \in G^+$ is represented in the Iwasawa  decomposition as 
\begin{align*}
 g &= d_\lambda n_u a_t k_\theta, \\
 d_\lambda &= \begin{pmatrix} \lambda &\\ & \lambda\end{pmatrix}, \quad n_u = \begin{pmatrix} 1 &\\ u & 1 \end{pmatrix}, \quad a_t = \begin{pmatrix} t&\\ & \frac{1}{t} \end{pmatrix}, \quad k_\theta = \begin{pmatrix} \cos(2\pi \theta) & \sin(2\pi \theta)\\ -\sin(2\pi \theta) & \cos(2\pi \theta) \end{pmatrix}
\end{align*}
and $f=a v^3 + bv^2w + cvw^2 + dw^3 = g \cdot x_{\pm}$, the discriminant has size $\lambda^{12}$, and the leading coefficient is $a = \frac{\lambda^3 t^3 \sin(2\pi \theta)}{\sqrt{2}}$ in the case of negative discriminant, $a = \frac{\lambda^3 t^3 \sin(6\pi \theta)}{(108)^{\frac{1}{4}}}$ in the case of positive discriminant.  If $a \neq 0$ then $|a| \geq 1$, so $t \gtrsim |\Disc(f)|^{\frac{-1}{12}}$.  To put this in the perspective of the familiar hyperbolic upper half plane $\Gamma \backslash \SL_2(\bR)/\SO_2(\bR)$, map $n_u a_t k_\theta \mapsto ((n_u a_t k_\theta)^{-1})^t$.  Then the $y$ coordinate in the hyperbolic plane is of order $t^2$, and among rings of discriminant $< X$, there are no lattice points with imaginary part greater than $\gtrsim X^{\frac{1}{6}}$, a region of hyperbolic volume $X^{-\frac{1}{6}}$.  This corresponds to a secondary main term in the Davenport-Heilbronn Theorem of order $X^{\frac{5}{6}}$ in the count of cubic fields of discriminant at most $X$.  

The above discussion applies to binary cubic forms for which $a \neq 0$.  If $a = 0$ the form is reducible, and for the form to be non-singular, it is now necessary that $b \neq 0$ so $|b| \geq 1$.  This implies the bound  $t \gg \frac{1}{|\Disc(f)|^{\frac{1}{4}}}$.  This gives a complete list of the barriers to equidistribution of this type.  One pole in the Eisenstein series twisted zeta function corresponds to each type of barrier to equidistribution.  Also, since $\E_r(n_u^t a_t) \ll t$, and since the number of cubic fields up to isomorphism with discriminant at most $X$ grows linearly in $X$ by Shintani's work, this guarantees the absolute convergence in $\RE(s)> \frac{5}{4}$ of the series defining the twisted zeta functions.

The argument in the Eisenstein case here splits the Eisenstein series into its constant term and non-constant term.  The non-constant term has rapid decay in the cusp, and can be handled in a similar way to the cuspidal case handled in \cite{H17}.  The constant term part is handled in a way similar to the original article of Shintani \cite{S72}, with an adjustment made to evaluate the residues.  In order to study the reducible forms, we follow Shintani in identifying this space with the space of binary quadratic forms \cite{S75}.  One pole in this case is matched against a pole of the whole zeta function.  The fact that the irreducible zeta function continues holomorphically to $\RE(s) > \frac{11}{12}$ is sufficient to obtain a power-saving error term in Weyl sums for the Eisenstein series part of the spectrum.

In \cite{L19} the second author determined the poles and residues of the double Dirichlet series enumerating the first and second covariants of a binary cubic form. The twisted zeta function here enumerates a quantity similar to the first and fourth covariants.  It is still of interest to study the analytic properties of a generating function for a multiple Dirichlet series enumerating three or more of the covariants.

\subsection*{Notation and conventions}
We abbreviate the contour integral \begin{equation}\frac{1}{2\pi i} \int_{c-i\infty}^{c + i \infty} F(z) dz = \oint_{\RE(z) = c} F(z)dz.\end{equation} Denote $e(x)=e^{2\pi i x}$, $c(x) = \cos(2\pi x)$, $s(x) = \sin(2\pi x)$.
The argument uses the following pair of standard Mellin transforms.  Write $K_\nu$ for the $K$-Bessel function.
For $\RE(s)> |\RE \nu|$, (\cite{I02}, p.205)
\begin{equation}
 \int_0^\infty K_\nu(x)x^{s-1}dx = 2^{s-2}\Gamma\left(\frac{s + \nu}{2}\right)\Gamma\left(\frac{s-\nu}{2}\right).
\end{equation}
We use the formula
\begin{equation}
 K_{\frac{s}{2}}(2) = \int_0^\infty e^{-t^2 -\frac{1}{t^2}} t^{s-1} dt.
\end{equation}
For $0 < \RE(s) < 1$, (\cite{B53}, p.13)
\begin{equation}
 \int_0^\infty \cos(x)x^{s-1}dx = \Gamma(s)\cos\left( \frac{\pi}{2}s\right).
\end{equation}
For functions $f$ on Euclidean space $\bR^d$, and $t \in \bR^\times$ we use the notation $f^t(x) = f(t\cdot x)$.  Under Fourier transform, this satisfies $\widehat{f^t}(\xi)= \frac{1}{t^d} \hat{f}\left(\frac{\xi}{t}\right)$.

The following  groups are used. 
\begin{itemize}
 \item $G_\bR = \GL_2(\bR)$
 \item $G^1 = \SL_2(\bR)$
 \item $G^+ = \{g \in \GL_2(\bR): \det g>0\}$
 \item $G_\zed = \GL_2(\zed)$
 \item $\Gamma = \SL_2(\zed)$
 \item $\Gamma_\infty = \left\{ \begin{pmatrix} \pm 1 &\\ n & \pm 1\end{pmatrix} : n \in \zed\right\}$
 \item $A = \{a_t: t \in \bR_{>0}\},$ $ a_t = \begin{pmatrix} t &\\ & \frac{1}{t}\end{pmatrix}$
 \item $N = \{n_x: x \in \bR\},$ $n_x = \begin{pmatrix} 1 &0 \\ x &1 \end{pmatrix}$
% \item $N' = \{\nu_x: x \in \bR\},$ $\nu_x =\begin{pmatrix} 1 &x\\0&1 \end{pmatrix}$
 \item $K = \{k_\theta: \theta \in \bR/ \zed\},$ $k_\theta = \begin{pmatrix} \cos 2\pi \theta & \sin 2\pi\theta \\ -\sin 2\pi\theta & \cos 2\pi\theta\end{pmatrix}$
 \item $B = \left\{\begin{pmatrix} b_{11} &0\\b_{21} & b_{22}\end{pmatrix}: b_{11}, b_{21}, b_{22} \in \bR, b_{11}b_{22} \neq 0 \right\}$, $B^+ = \{b \in B: b_{11}, b_{22}>0\}$
 \item $d_\lambda = \begin{pmatrix} \lambda &\\ & \lambda \end{pmatrix}$.
\end{itemize}

We follow Shintani's conventions \cite{S72} regarding integrals and automorphic forms on $\SL_2(\bR)$. The Iwasawa decomposition is $G = KAN$ with Haar measure, for $f \in L^1(G^1)$,
\begin{align}
 \int_{G^1} f(g) dg &=  \int_0^{1}\int_{-\infty}^\infty \int_0^\infty f(k_\theta a_t n_u)\frac{dt}{t^3}du d\theta
\end{align}
and for $f \in L^1(G^+)$,
\begin{equation}
 \int_{G^+}f(g)dg = \int_0^\infty \int_{G^1} f\left(\begin{pmatrix} \lambda &\\ &\lambda\end{pmatrix} g \right)dg \frac{d\lambda}{\lambda}.
\end{equation}
Given a group element $g$, write $k(g), t(g), u(g)$ for the elements of $K, A, N$ in the representation of $g$ in the Iwasawa decomposition.  
The Siegel set $\fS_C$ is
\begin{equation}
 \fS_C = \left\{k_\theta a_t n_u: \theta \in \bR/\zed, t \geq C, |u| \leq \frac{1}{2}\right\}.
\end{equation}
For any $r \in \bR$, define the semi-norm 
\begin{equation}
 \mu(r)(f) = \sup_{g \in \fS_{\frac{1}{2}}} t(g)^r |f(g)|.
\end{equation}
Let $C(G^1/\Gamma, r) = \{f \in C(G^1/\Gamma), \mu(r)(f)<\infty\}$.

Shintani's normalization of the Eisenstein series makes this right $\Gamma$ and left $K$ invariant, 
\begin{equation}\label{def_real_analytic_eisenstein}
E(z,g) = \sum_{\gamma \in \Gamma/\Gamma_\infty}
t(g \gamma)^{z+1}.
\end{equation}
 The function $E(z,g)$ satisfies the functional equation
\begin{equation}
\xi(1+z )E(z,g) = \xi(1-z)E(-z,g); \qquad \xi(z) =
\pi^{-\frac{z}{2}}\Gamma\left(\frac{z}{2}\right)\zeta(z)
\end{equation}
and has a Fourier development in $z \neq 0$ given by
\begin{align}\label{eisenstein_fourier_dev}
E(z,g)&= t^{1+z}+ t^{1-z}\frac{\xi(z)}{\xi(z+1)}\\\notag & \qquad+
\frac{4t}{\xi(z+1)} \sum_{m=1}^\infty
\eta_{\frac{z}{2}}(m)K_{\frac{z}{2}}(2\pi m t^2)\cos 2\pi m u,\\\notag
\eta_{\frac{z}{2}}(m)&=\sum_{ab = m}\left(\frac{a}{b} \right)^{\frac{z}{2}},
\end{align}
with $K_\nu$ the $K$ Bessel function.  We use frequently that $E(z, g) = E(z, (g^{-1})^t)$. The Riemann $\xi$ function $\xi(z)$ satisfies the functional equation $\xi(z) = \xi(1-z)$.

We also use the incomplete Eisenstein series to regularize integrals in the same way as Shintani. Let $\Psi$ denote the space of entire
functions such that for all $\psi \in \Psi$, for all $-\infty < C_1 < C_2 <
\infty$, for all $ N >0$,
\begin{equation}
 \sup_{C_1 < \RE (w) < C_2}\left(1 + (\IM w)^2\right)^N |\psi(w)| < \infty.
\end{equation}
For $\psi \in \Psi$ and $\RE(w) > 1$ choose $1 < c < \RE (w)$ and set
\begin{equation}
\sE(\psi, w; g) = \oint_{\RE (z) = c}\psi(z)\frac{E(z,g)}{w-z}dz.
\end{equation}
Shintani Lemma 2.9 gives the following estimates.
\begin{lemma}
 We have
 \begin{enumerate}
  \item $\sE(\psi, w; g) \in C(G^1/\Gamma, \RE w-1)$
  \item For a fixed $\psi$,
  \begin{equation}
   \sup_{1 \leq w \leq M, g \in \fS_{\frac{1}{2}}} |(w-1)\sE(\psi, w;g)| < \infty, (M>1)
  \end{equation}
  \item $\lim_{w \to 1^+} (w-1)\sE(\psi, w;g) = \frac{\psi(1)}{\xi(2)}.$
 \end{enumerate}

\end{lemma}

The corollary to Lemma 2.9 in \cite{S72} states that, for $f \in L^1(G^1/\Gamma, dg)$,
\begin{equation}
 \lim_{w \to 1^+} (w-1) \int_{G^1/\Gamma} f(g) \sE(\psi, w;g)dg = \frac{\psi(1)}{\xi(2)} \int_{G^1/\Gamma} f(g)dg.
\end{equation}
Similarly, 
\begin{lemma}
For $0 < c < w$ and $f \in L^1(G^1/\Gamma_\infty)$,
\begin{equation}
 \lim_{w \to 0^+} w \int_{G^1/\Gamma_\infty} f(g) \oint_{\RE(\alpha) = c}\frac{\psi(\alpha)t(g)^\alpha}{\alpha (w-\alpha)} d\alpha dg= \psi(0)\int_{G^1/\Gamma_\infty} f(g)dg.
\end{equation}
 
\end{lemma}
\begin{proof}
Let $w< \frac{1}{2}$,  $F_w(t) = \oint_{\RE(\alpha) = c}\frac{\psi(\alpha)t^\alpha}{\alpha (w-\alpha)} d\alpha$.
 Let $\epsilon > 0$.  For $t > \epsilon$, shift the contour left to the line $\RE(\alpha) = -1$, where the integral is uniformly bounded in $w$ and $\epsilon$.  A pole is passed at 0 with residue $\frac{\psi(0)}{w}$.  If $t \leq \epsilon$, shift the contour right to $\RE(\alpha) = 1$, passing a pole at $w$ with residue $\frac{\psi(w)t^w}{w}$ with an integral that is uniformly bounded in $w$.  Letting $w \to 0$ obtains the claim.
 
\end{proof}

Let $f \in C_c^\infty(G^1)$ be bi-$K$-invariant, that is, for any $g \in G^1$ and $k_{\theta_1}, k_{\theta_2}$, $f(g) = f(k_{\theta_1}gk_{\theta_2})$. 
For imaginary $z = i \gamma$, $E(i\gamma, g) = \E_{\frac{1 + \gamma^2}{4}}(g^{t})$, which is left invariant under $\Gamma$. Let $\E_r^c$ be the constant term in the Fourier expansion, and $\E_r^n = \E_r - \E_r^c$.
As a right convolution operator $f$ acts on $\E_r$ as multiplication by a scalar.  To check this, note that
\begin{equation}
 \E_r*f (g_0) = \int_{gh = g_0} \E_r(g)f(h)dh
\end{equation}
is left $\Gamma$ invariant and right $K$ invariant.  Also, it is an eigenfunction of the Laplacian and Hecke operators, with the same eigenvalues as $\E_r$.  It follows by multiplicity one that the convolution is a multiple of $\E_r$. The following lemma determines the eigenvalue.

\begin{lemma}
 We have
 \[
  \E_r * f = \left( \int_{G^1} f(g) t(g)^{1+z}dg\right) \E_r. 
 \]
For the choice $f(g) = \exp\left(-\tr g^t g\right)$ the eigenvalue is $\sqrt{\pi}K_{\frac{z}{2}}(2)$.
\end{lemma}

\begin{proof}
 Let $\psi \in C_c^\infty(\Gamma \backslash G^1/K)$ be a smooth test function and let $\psi_0$ be the constant term in its Fourier expansion in the parabolic direction.  The Petersson inner product of $\E_r$ with $\psi$ is a Mellin transform of $\psi_0$,
 \begin{align*}
  \int_{\Gamma \backslash G^1} \E_r(g) \overline{\psi(g)} dg &= \int_{\Gamma \backslash G^1} \sum_{\gamma \in \Gamma_\infty \backslash \Gamma} t(\gamma g)^{1+z} \overline{\psi(\gamma g)} dg\\
  &= \int_{\Gamma_\infty \backslash G^1} t(g)^{1 + z} \overline{\psi(g)} dg\\
  &= \int_0^\infty \psi_0(t) t^{-1 + z} \frac{dt}{t} = \tilde{\psi}_0(-1+z).
 \end{align*}
Next we calculate the inner product with the convolution $\E_r * f$,
\begin{align*} 
 \int_{\Gamma \backslash G^1} (\E_r * f)(g) \overline{\psi(g)} dg &= \int_{\Gamma \backslash G^1} \int_{G^1} \E_r(h) f(h^{-1}g) \overline{\psi(g)} dh dg\\
 &= \int_{\Gamma \backslash G^1} \int_{G^1} \sum_{\gamma \in \Gamma_\infty \backslash \Gamma} t(h)^{1+z} f(h^{-1}\gamma g) \overline{\psi(\gamma g)} dh dg\\
 &= \int_{\Gamma_\infty \backslash G^1} \int_{G^1} t(h)^{1+z} f(h^{-1}g) \overline{\psi(g)}dh dg\\
 &= \int_0^\infty \frac{dt_1}{t_1^3} \int_0^\infty \frac{dt_2}{t_2^3} \int_{-\infty}^\infty du t_1^{1+z} f\left(\begin{pmatrix} \frac{t_2}{t_1} & \frac{u}{t_1t_2}\\ 0 & \frac{t_1}{t_2} \end{pmatrix}\right) \psi_0(t_2).
\end{align*}
After a change of coordinates we obtain
\[
 \tilde{\psi}_0(-1+z) \int_{0}^\infty \frac{dt}{t} t^{z} \int_{-\infty}^\infty du f\left(\begin{pmatrix} \frac{1}{t} & u \\ 0 & t\end{pmatrix}\right) = \tilde{\psi}_0(-1+z) \int_{0}^\infty \frac{dt}{t} t^{1+z} \int_{-\infty}^\infty du f\left(n_u a_t\right).
\]
Let $f(g) = \exp(-\tr(g^tg))$ so that the eigenvalue may be written
\[
 \int_{0}^\infty \frac{dt}{t} t^{z} \int_{-\infty}^\infty du \exp\left(-t^2 - \frac{1}{t^2} - u^2\right).
\]
The integral is $\sqrt{\pi}K_{\frac{z}{2}}(2)$.
\end{proof}

 \section{Cubic rings, binary cubic forms}

A cubic ring $R$ over $\zed$ is a free rank three $\zed$ module with a ring multiplication.  Delone-Fadeev and Gan-Gross-Savin established a discriminant-preserving bijection between cubic rings up to isomorphism and the space $\sym^3(\zed^2)$ of binary cubic forms. In the identification, maximal cubic rings whose associated form is irreducible over $\zed$ correspond with rings of integers in cubic number fields.

Given a form $x (v,w) = av^3 + bv^2w + cvw^2 + dw^3$ in the space $V_{\bR}$ of real binary cubic forms, $g \in  \GL_2(\bR)$ acts by $g\cdot x(v,w) = x((v,w)g)$. There is a bilinear pairing 
\[
 \langle x, y\rangle = x_4y_1 - \frac{1}{3}x_3y_2 + \frac{1}{3}x_2y_3 - x_1y_4.
\]
Let $P(x) = x_2^2x_3^2 +18x_1x_2x_3x_4 -4x_1x_3^3 -4x_2^3x_4 -27x_1^2x_4^2$ be the discriminant.  
The discriminant scales under the action by a factor of $\chi(g)=\det(g)^6$. The space $V_{\bR}$ decomposes into two open orbits $V_+$ and $V_-$ having positive and negative discriminant, and the singular set $S$ where the discriminant is 0.   
We identify $V_+ = G^+ \cdot x_+$, $V_- = G^+ \cdot x_-$ by choosing base points
\begin{equation}
 x_+ = \left(0,\frac{3}{(108)^{\frac{1}{4}}},0,\frac{-1}{(108)^{\frac{1}{4}}}\right), \qquad x_- = \left(0,\frac{1}{\sqrt{2}},0,\frac{1}{\sqrt{2}}\right).
\end{equation}
The point $x_+$ has stabilizer of order 3 generated by the rotation of $\frac{2\pi}{3}$, while $x_-$ has trivial stabilizer.  We have the group integrals (\cite{S72}, Proposition 2.4)
\[
 \int_{g \in G^+} f(g \cdot x_+)dg = \frac{1}{4\pi} \int_{V_+} f(x) \frac{dx}{P(x)}
\]
and
\[
 \int_{g \in G^+} f(g \cdot x_-) dg = \frac{1}{12\pi} \int_{V_-} f(x) \frac{dx}{|P(x)|}.
\]
In particular,
\begin{align*}
 \int_{V_+} f(x) dx &= \frac{1}{4\pi} \int_{g \in G^+} f(g \cdot x_+) \chi(g) dg,\\
 \int_{V_-} f(x) dx &= \frac{1}{12\pi} \int_{g \in G^+} f(g \cdot x_-) \chi(g) dg.
\end{align*}
The Fourier transform is defined by
\begin{align*}
 \hat{f}(\xi) = \int_{V_\bR} f(x) e^{-2\pi i \langle x, \xi\rangle}dx.
\end{align*}

In this paper it suffices to restrict attention to test functions which are bi-$K$-invariant and factor through the determinant.  Let $f_G$ be the function on $G^1$, $f_G(g) = \exp\left(-\tr g^t g\right)$, and extend $f_G$ to $G^+$ independent of the determinant.  Let $f_D \in C_c^\infty(\bR^+)$.  Define $f_-$ supported on $V_-$ and $f_+$ supported on $V_+$ by
\[
 f_-(g\cdot x_-) = f_G(g)f_D(\chi(g)), \qquad f_+(g \cdot x_+) = f_G(g)f_D(\chi(g)).
\]
For such functions, the Fourier transform $\hat{f}$ is left-$K$-invariant, since
\begin{align*}
 \hat{f}(k_\theta \cdot \xi) &= \int_{V_{\bR}} f(x) e^{-2\pi i \langle x, k_{\theta}\cdot \xi \rangle} dx\\
 &= \int_{V_{\bR}} f(x) e^{-2\pi i \langle k_{-\theta} x, \xi\rangle} dx = \hat{f}(\xi).
\end{align*}

The following integrals of the Fourier transform of $f_{\pm}$ are used.

Define $\Sigma_2(f, z) = \int_{0}^\infty f(0,0,0,t)t^{z-1} dt.$  This satisfies \begin{equation}\Sigma_2(f^t, z) = t^{-z} \Sigma_2(f,z).\end{equation}
\begin{lemma}\label{Sigma_2_lemma} We have 
 \begin{align*} 
  \Sigma_2\left(\hat{f}_-,z\right) &= 2^{-\frac{z}{2}} \pi^{1-z} \cos\left(\frac{\pi z}{2} \right)\frac{\Gamma(z)\Gamma\left(\frac{1+z}{2} \right)}{\Gamma\left(1 + \frac{z}{2}\right)}\tilde{f}_D\left(1 - \frac{z}{4}\right) K_{\frac{1-3z}{2}}(2),\\
  \Sigma_2\left(\hat{f}_+,z\right)& = 3^{\frac{3z}{4}-1}2^{-\frac{z}{2}} \pi^{1-z} \cos\left(\frac{\pi z}{2} \right)\frac{\Gamma(z)\Gamma\left(\frac{1+z}{2} \right)}{\Gamma\left(1 + \frac{z}{2}\right)}\tilde{f}_D\left(1 - \frac{z}{4}\right) K_{\frac{1-3z}{2}}(2).
 \end{align*}
\end{lemma}
\begin{proof}
 Note that the left-$K$-invariance causes $f_{\pm}$ to be even. Calculate
 \begin{align*}
  \Sigma_2\left(\hat{f}_{\pm}, z\right) &= \int_0^\infty \hat{f}_{\pm}(0,0,0,t)t^{z}\frac{dt}{t}\\
  &= \int_0^\infty t^z \frac{dt}{t} \int_{x \in V_{\bR}} f_{\pm}(x) \cos(2\pi  x_1 t)\\
  &= (2\pi)^{-z}\cos\left(\frac{\pi z}{2}\right) \Gamma(z)\int_{x \in V_{\bR}} f_{\pm}(x) |x_1|^{-z}.
  \end{align*}
  In the case of $f_-$, write this as 
  \begin{align*}
  &(2\pi)^{-z}\cos\left(\frac{\pi z}{2}\right) \Gamma(z) \Bigl(12\pi \int_{g \in G^+}f_G(g) |(g\cdot x_-)_1|^{-z} f_D(\chi(g))\chi(g) dg \Bigr) .
 \end{align*}
In the $ANK$ decomposition, Haar measure is $\frac{dt}{t}du d\theta$.  In the negative discriminant case, $x_- = \frac{1}{\sqrt{2}}(0, 1, 0, 1)$ and the first coefficient of $a_t n_u k_\theta \cdot x_-$ is $\frac{t^3 \sin 2\pi \theta}{\sqrt{2}}$.  
Using that $f$ is right $K$ invariant, integrate in $\theta$ using 
\[
 \int_0^1 |\sin(2\pi \theta)|^{-z} d\theta = \frac{\Gamma\left( \frac{1+z}{2} \right)}{\sqrt{\pi}\Gamma\left(1 + \frac{z}{2} \right)}.
\]
Thus the negative discriminant case is given by
\begin{align*}
 &2^{\frac{z}{2}}(2\pi)^{-z} \cos\left(\frac{\pi z}{2} \right) \Gamma(z) 12 \pi \\&\times \int_0^\infty \frac{d\lambda}{\lambda} f_D(\lambda^{12}) \lambda^{12-3z} \int_0^1 |\sin(2\pi \theta)|^{-z} d\theta \int_{-\infty}^\infty du \int_0^\infty \frac{dt}{t} t^{-3z} \exp\left(-t^2 -\frac{1}{t^2} -\frac{u^2}{t^2} \right)\\
 &= 2^{-\frac{z}{2}} \pi^{1-z} \cos\left(\frac{\pi z}{2} \right)\frac{\Gamma(z)\Gamma\left(\frac{1+z}{2} \right)}{\Gamma\left(1 + \frac{z}{2}\right)}\tilde{f}_D\left(1 - \frac{z}{4}\right) K_{\frac{1-3z}{2}}(2).
\end{align*}

In the case of $f_+$, write the integral as
  \begin{align*}
  &(2\pi)^{-z}\cos\left(\frac{\pi z}{2}\right) \Gamma(z) \Bigl(4\pi \int_{g \in G^+}f(g) |(g\cdot x_+)_1|^{-z} \chi(g) dg \Bigr) .
 \end{align*}
In the positive discriminant case, use that the stabilizer of $x_+$ is the rotation group generated by rotation by $\frac{2\pi}{3}$.  The first coefficient of $a_t n_u k_\theta \cdot x_+$ is $\frac{t^3 \sin 6\pi \theta }{(108)^{\frac{1}{4}}}$.  Thus the positive discriminant case is given by $3^{\frac{3z}{4}-1}$ times the integral in the negative discriminant case.  This obtains the lemma.

\end{proof}

Define 
\begin{equation}
 \Sigma_3(f, s) = \int_0^\infty dt\int_{-\infty}^\infty du f(0,0, t,u)t^{s-1}.
\end{equation}
This satisfies \begin{equation}\Sigma_3\left(f^t, s\right) = t^{-s-1}\Sigma_3(f,s).\end{equation}

\begin{lemma}\label{Sigma_3_lemma}
 Let, for $x \in \bR^3$, $f_{\pm,0}(x) = f_{\pm}(0,x)$. We have
 \[
  \Sigma_3\left(\hat{f}_{\pm}, s\right) = 3^s\pi^{-s+\frac{1}{2}} \frac{\Gamma\left(\frac{s}{2}\right)}{2\Gamma\left(\frac{1-s}{2}\right)} \int_{x = (x_2, x_3, x_4)} f_{\pm,0}(x) |x_2|^{-s}.
 \]

\end{lemma}

\begin{proof}
 Using the bilinear pairing $\langle x, y\rangle= x_4y_1 - \frac{1}{3}x_3y_2 +\frac{1}{3}x_2y_3 - x_1y_4$, and the fact that $f$ is even, calculate
 \begin{align*}
  \Sigma_3\left(\hat{f}_{\pm},s\right) &= \int_{x = (x_2, x_3, x_4)} \int_0^\infty \frac{dt}{t} f_{\pm}(0, x_2, x_3, x_4)t^s \cos\left(2\pi  \frac{tx_2}{3}\right)dx\\
  &= 3^{s}(2\pi)^{-s}\cos\left(\frac{\pi s}{2}\right)\Gamma(s) \int_{x = (x_2, x_3, x_4)} f_{\pm,0}(x) |x_2|^{-s} dx.
 \end{align*}
Combined with the formula $\cos \left(\frac{\pi s}{2}\right) = \frac{\pi}{\Gamma\left(\frac{s+1}{2}\right)\Gamma\left(\frac{1-s}{2}\right)}$ and the formula 
\[
 \Gamma(s) = \frac{\Gamma\left(\frac{s}{2}\right)\Gamma\left(\frac{s+1}{2}\right)}{2^{1-s}\sqrt{\pi}}
\]
this proves the lemma.
\end{proof}

\begin{lemma} \label{Phi_0_lemma}
 We have
 \begin{align*}
  \int_{x= (x_2,x_3, x_4)}f_{-,0}(x)|x_2|^s dx &= 2^{\frac{-1-s}{2}} \tilde{f}_D\left(\frac{3 + s}{4} \right)\sqrt{\pi} K_{\frac{s-2}{2}}(2),\\
  \int_{x = (x_2, x_3, x_4)} f_{+,0}(x)|x_2|^s dx &=3^{\frac{s-1}{4}}2^{\frac{-1-s}{2}} \tilde{f}_D\left(\frac{3 + s}{4} \right)\sqrt{\pi} K_{\frac{s-2}{2}}(2).
 \end{align*}

\end{lemma}
\begin{proof}
 Let $g_2 \cdot x$ denote the action of $G^+$ on binary quadratic forms.  Identify $x_+, x_-$ with points in the space of binary quadratic forms by dropping the first coefficient.  The action of $(d_\lambda a_t n_u)_2$ on $x_{\pm}$ is given by
 \begin{align*}
  (d_\lambda a_t n_u)_2 \cdot x_- &= \left( \frac{\lambda^2 t^2}{\sqrt{2}}, \frac{2\lambda^2 u}{\sqrt{2}}, \frac{\lambda^2 (1 + u^2)}{\sqrt{2}t^2}\right),\\
  (d_\lambda a_t n_u)_2 \cdot x_+ &= \left(\frac{3\lambda^2t^2}{(108)^{\frac{1}{4}}}, \frac{6 \lambda^2 u}{(108)^{\frac{1}{4}}}, \frac{\lambda^2(-1 + 3u^2)}{(108)^{\frac{1}{4}}t^2} \right).
 \end{align*}
 In the case of $V_-$, the volume form is equal to \[|dx_1 \wedge dx_2 \wedge dx_3| = \frac{\lambda^5 }{2^{\frac{5}{2}}t} |dt \wedge du \wedge d\lambda|,\] while on $V_+$ the volume form is given by
 \[
  |dx_1 \wedge dx_2 \wedge dx_3| = \frac{1}{3^{\frac{1}{4}}} \frac{1}{2^{\frac{5}{2}}} \frac{\lambda^5}{t} |dt \wedge du \wedge d\lambda|.
 \]
Thus the integral over $V_-$ is given by
\begin{align*}
& 2^{\frac{5-s}{2}} \int_{0}^\infty \frac{d\lambda}{\lambda} \lambda^{6+2s} \int_0^\infty \frac{dt}{t} t^{2s} \int_{-\infty}^\infty du  f_{-,0}((d_\lambda a_t n_u)_2 \cdot x_-)\\
&= 3\cdot 2^{\frac{3-s}{2}} \int_0^\infty \frac{d\lambda}{\lambda} \lambda^{9 + 3s} \int_0^\infty \frac{dt}{t} t^{-3+s} \int_{-\infty}^\infty du f_{-,0} \left( \left( d_{\sqrt{\frac{\lambda^3}{t}}} a_t n_u\right)_2 \cdot x_-\right)\\
&=3\cdot 2^{\frac{3-s}{2}} \int_0^\infty \frac{d\lambda}{\lambda} \lambda^{9 + 3s} \int_0^\infty \frac{dt}{t} t^{-3+s} \int_{-\infty}^\infty du f_- \left( \left( d_{\lambda} a_t n_u\right)_3 \cdot x_-\right)\\
&= 3 \cdot 2^{\frac{3-s}{2}} \int_0^\infty \frac{d\lambda}{\lambda} \lambda^{9 + 3s} f_D(\lambda^{12}) \int_0^\infty \frac{dt}{t} t^{-3+s} \int_{-\infty}^\infty du \exp\left(-t^2 - \frac{1}{t^2} - \frac{u^2}{t^2}\right)\\
&= 2^{\frac{-1-s}{2}} \tilde{f}_D\left(\frac{3 + s}{4} \right)\sqrt{\pi} K_{\frac{s-2}{2}}(2).
\end{align*}
The integral over $V_+$ differs from the above by a factor of $3^{\frac{s-1}{4}}$.
\end{proof}

Write $L$ for space of integral binary cubic forms and $\hat{L}$ for the dual forms, which have middle coefficients divisible by 3. For each $m  \neq 0$ let $h(m)$ be the class number of forms of discriminant $m$ and $\hat{h}(m)$ the class number of dual forms.  
Shintani obtained the following lemma regarding singular integral forms.
\begin{lemma}\label{fibration_lemma} The singular forms  $\hat{L}_0$ are the disjoint union
\begin{align}
\hat{L}_0 &= \{0\}\sqcup\bigsqcup_{m=1}^\infty \bigsqcup_{\gamma \in
	\Gamma/\Gamma \cap N} \{\gamma \cdot (0,0,0,m)\}\sqcup \bigsqcup_{m=1}^\infty
\bigsqcup_{n=0}^{3m-1}\bigsqcup_{\gamma \in \Gamma}\{\gamma \cdot (0,0, 3m,n)\}.
\end{align}
Let
\begin{align}L_0(I) &= \bigsqcup_{m=1}^\infty \bigsqcup_{\Gamma/\Gamma\cap
	N}\{\gamma
\cdot (0,0,0,m)\},\\ \notag
\hat{L}_0(II) &= \bigsqcup_{m=1}^\infty
\bigsqcup_{n=0}^{3m-1}\bigsqcup_{\gamma \in \Gamma}\{\gamma \cdot (0,0,
3m,n)\}.\end{align}
\end{lemma}

The forms of positive discriminant break into two classes, the first of which
have stability group in $\Gamma$ which is trivial, and the second having
stability group of order 3.

For $m \neq 0$ let $\{g_{i,m}\}_{1 \leq i \leq h(m)} \subset G^+,$ (resp.\ $\{\hat{g}_{i,m}\}_{1 \leq i \leq \hat{h}(m)}$) be such that 
 \begin{equation}\{x_{i,m} = g_{i,m}\cdot x_{\sgn \;m}\}_{1 \leq i \leq h(m)}\end{equation} (resp.\ $\{\hat{x}_{i,m} = \hat{g}_{i,m}\cdot x_{\sgn \;m}\}_{1 \leq i \leq \hat{h}(m)}$)  are
representatives for the classes of binary cubic forms of discriminant $m$ (resp.\ classes of dual forms). The points $g_{i,m}$ naturally identify the cubic ring associated with $x_{i,m}$ with its lattice shape  in $\Gamma \backslash G^1/K$.    Set $\Gamma(i,m)< \Gamma$ the
stability group of $x_{i,m}$, similarly $\hat{\Gamma}(i,m)$.

\subsection{Reducible forms}

Due to the growth of the Eisenstein series in the cusp, and the fact that a significant contribution of forms reducible over $\zed$ occur in the cusp, in this work it is necessary to give special treatment to the reducible forms.  Evidently every non-degenerate reducible form $x$ is equivalent under $\Gamma$ to a form of type $\sR = \{bv^2w + cvw^2 + dw^3: 0 \leq c < 2b\}$.  If $c^2 - 4bd$ is not a square, then $f$ has a unique representative in $\sR$, while if $c^2 - 4bd$ is a square then there are one or three forms equivalent to $x$ in $\sR$ according as the stabilizer subgroup of $x$ in $\Gamma$ has size three or one, see the discussion in \cite{S75}, pp.\ 45-46.

Following Shintani, \cite{S75}, our treatment of reducible forms identifies $\sR$ with the prehomogeneous vector space $\sym^2(\bR^2)$ of binary quadratic forms acted on by the group of lower triangular matrices \[B^+ = \left\{ \begin{pmatrix} g_{11}& \\g_{21} & g_{22} \end{pmatrix}: g_{11}, g_{22} > 0\right\}.\]
Given a binary quadratic form $x = av^2 + bvw + cw^2$ associated to symmetric matrix $Q = \begin{pmatrix} a & \frac{b}{2}\\ \frac{b}{2} & c\end{pmatrix}$ the action of $g \in B^+$ is written $\rho(g) \cdot x$ or $g \cdot x$ for short, and maps
\[
 Q \mapsto gQg^t.
\]
There are two invariants, the discriminant $D = b^2 - 4ac$ and the first coefficient $f_1=a$.  Let $\chi(g) = (\det g)^2$ and $\chi_1(g) = g_{11}^2$.  We have \[\Disc( g \cdot x) = \chi(g) \Disc(x), \qquad (g \cdot x)_1 = \chi_1(g) x_1.\] The contragredient representation of $\rho$ is $\rho^*(g) = \frac{1}{\det(g)^2} \rho(g)$.
When $g \in B^+$ is represented in coordinates $g = \begin{pmatrix} t_1 & \\ u & t_2 \end{pmatrix}$, Shintani uses the Haar measure  $dg = t_1^{-2} t_2^{-1} dt_1dt_2 du$ which satisfies 
\begin{equation}
\int_{B^+} f\left(\begin{pmatrix} t_1 & \\ u & t_2 \end{pmatrix} \right) \frac{dt_1}{t_1^2}\frac{dt_2}{t_2}du = 2 \int_0^\infty \frac{d\lambda}{\lambda}\int_0^\infty \frac{dt}{t^3} \int_{-\infty}^\infty du f(d_\lambda a_t n_u).
\end{equation}
We use the latter normalization, so that the orbital zeta functions in our work differ by a factor of $\frac{1}{2}$ from Shintani's.
There is a bilinear pairing on $\sym^2(\bR^2)$ given by $[x,y] = x_1y_3 - \frac{1}{2} x_2y_2 + x_3y_1$. Let $L = \sym^2(\zed^2)$ with dual forms $L^* = \sym^2(\zed^2)^*$.  Thus dual integral forms have even middle coefficient.

The singular points of the representation $(V, B^+, \rho)$ are the union $S \cup S_1$ where $S = \{\Disc = 0\}$, $S_1 = \{x_1 = 0\}$. Let $L_{0,r}$ be the singular integral forms, and $\hat{L}_{0,r}$ the singular dual integral forms.
\begin{lemma}\label{reducible_singular_fibration_lemma}
 The set $\hat{L}_{0,r}$ is the disjoint union
 \begin{align}
  &\{(0,0,0)\} \sqcup \{(0,0,\ell): \ell \in \zed \setminus \{0\}\} \sqcup B_\zed^+ \cdot \{(0, 2m, n): m \neq 0, 0 \leq n <|2m|\} \\ \notag &\sqcup B_\zed^+ \{\ell(b^2, 2bd, d^2): (b,d) = 1, \ell \neq 0, 0 \leq d < b\}.\end{align}
  Let
  \begin{align*}
 L_{0,r}(I)&= \{(0,0,\ell): \ell \in \zed \setminus \{0\}\}\\
 \hat{L}_{0,r}(II) &= \bigsqcup_{m \neq 0} \bigsqcup_{0 \leq n < |2m|} \bigsqcup_{\gamma \in B_\zed^+}  \{\gamma \cdot (0, 2m, n)\}\\
 \hat{L}_{0,r}(III)&= \bigsqcup_{\ell \neq 0} \bigsqcup_{b \neq 0} \bigsqcup_{\substack{0 \leq d < b, \\(b,d)=1}} \bigsqcup_{\gamma \in B_{\zed}^+} \{\gamma \cdot\ell(b^2, 2bd, d^2)\}.
  \end{align*}

\end{lemma}

\begin{proof}
 See \cite{S75}, Lemma 4.
\end{proof}

Shintani \cite{S75} introduces the orbital zeta functions, for Schwarz class $f$ 
\begin{align*}
 Z(f, s_1, s_2) &= \int_{B^+/B_{\zed}^+} \chi_1(g)^{s_1}\chi(g)^{s_2} \sum_{x \in L} f(\rho(g) \cdot x) dg\\
 Z^*(f, s_1, s_2) &= \int_{B^+/B_{\zed}^+} \chi_1(g)^{s_1}\chi(g)^{s_2} \sum_{x \in L^*} f(\rho(g) \cdot x) dg.
\end{align*}
Let $V_{2,\pm}= \{x \in \sym^2(\bR): \pm \Disc(x) > 0\}$. Define 
\[
 \Phi_\pm(f,s_1,s_2) = \int_{V_{2, \pm}} f(x) |x_1|^{s_1} |\Disc(x)|^{s_2} dx.
\]
In the case that $f = f_{\pm, 0}$ from the previous section, and $s_2 = 0$, by Lemma \ref{Phi_0_lemma},
\begin{align*}
 3^s \pi^{-s + \frac{1}{2}} \frac{\Gamma\left(\frac{s}{2}\right)}{2\Gamma\left(\frac{1-s}{2}\right)}\Phi_+(f_{+,0}, s, 0) &= \Sigma_3\left(\hat{f}_+, -s\right),\\
 3^s \pi^{-s + \frac{1}{2}} \frac{\Gamma\left(\frac{s}{2}\right)}{2\Gamma\left(\frac{1-s}{2}\right)}\Phi_-(f_{-,0}, s, 0) &= \Sigma_3\left(\hat{f}_-, -s\right).
\end{align*}

Also,
\[
 \Sigma(f,s) = \int_{-\infty}^\infty \int_{-\infty}^\infty f(t, 2u, t^{-1}u^2) |t|^{s-1}dt du.
\]
This is holomorphic in $\RE(s)>\frac{1}{2}$.

Let $L', {L^*}'$ denote the non-singular terms in $L, L^*$.  Define 
\begin{align*}
 Z_+(f, s_1, s_2) &= \int_{B^+/B_{\zed}^+, \chi(g) \geq 1} \chi_1(g)^{s_1}\chi(g)^{s_2} \sum_{x \in L'} f(\rho(g)\cdot x) dg\\
 Z_+^*(f, s_1, s_2) &= \int_{B^+/B_{\zed}^+, \chi(g) \geq 1} \chi_1(g)^{s_1} \chi(g)^{s_2}\sum_{x \in {L^*}'} f(\rho(g)\cdot x) dg.
\end{align*}

\begin{lemma}[\cite{S75}, Lemma 4]
 If $\RE s_1, s_2 > 1$ and $f \in \sS(V_{\bR})$,
 \begin{align*}
  Z(f, s_1, s_2) &= Z_+(f, s_1, s_2) + Z_+^*\left(\hat{f}, s_1, \frac{3}{2}-s_1-s_2\right)\\
  &+ \frac{1}{4}(2s_1 + 2s_2 -3)^{-1}\zeta(s_1)\zeta(2s_1-1)\zeta(2s_1)^{-1}\Sigma\left(\hat{f}, s_1-1\right)\\
  &+ \frac{1}{8}(s_2-1)^{-1}\zeta(s_1)(\Phi_+ + \Phi_-)(f, s_1-1,0)\\
  &- (8s_2)^{-1}\zeta(s_1)\zeta(2s_1-1)\zeta(2s_1)^{-1}\Sigma(f, s_1-1)\\
  &-\frac{1}{8}(2s_1 + 2s_2 - 1)^{-1}\zeta(s_1)(\Phi_+ + \Phi_-)(\hat{f}, s_1-1,0).
 \end{align*}

\end{lemma}
\begin{proof}
 This is Lemma 4 of \cite{S75}.  Note that the change of factor of two accounts for the difference between our normalization of Haar measure and Shintani's.
\end{proof}

The square discriminant terms are handled in the following lemma. Let
\[
 Z_{\square}(f, s_1, s_2) = \int_{B^+/B_\zed^+} \chi_1(g)^{s_1}\chi(g)^{s_2} \sum_{x \in L', \Disc(x) = \square} f(\rho(g)\cdot x).
\]

\begin{lemma}[\cite{S75}, Lemma 7]\label{square_discriminant_lemma}
 Define 
 \[
  \Xi(s_1, s_2) = \frac{\zeta(s_1)^2}{\zeta(2s_1)} \frac{\zeta(2s_1 + 2s_2-1)\zeta(2s_2)}{\zeta(s_1 + 2s_2)}.
 \]
Then
\[
 Z_{\square}(f, s_1, s_2) = \frac{1}{4}\Xi(s_1, s_2) \Phi_+(f, s_1-1, s_2-1).
\]

\end{lemma}
\begin{proof}
 The factor of 2 compared to \cite{S75} accounts for the difference in Haar measure.
\end{proof}

\section{Twisted $\sL$ functions}
As in the previous work \cite{H17}, we rely on automorphic twists of the zeta functions introduced by Shintani.

Define
\begin{align}
 \sL^+(\E_r, s) &= \sum_{m=1}^\infty \frac{1}{m^s}  \sum_{i=1}^{h(m)}  \frac{\E_r(g_{i,m})}{|\Gamma(i,m)|}\\
 \notag \sL^-(\E_r, s) &= \sum_{m=1}^\infty \frac{1}{m^s} \sum_{i=1}^{h(-m)} \E_r(g_{i,-m}).
\end{align}
As in the previous section, let $f_G$ be defined on $G^1$ by $f_G(g) = \exp\left(-\tr g^t g\right)$ and extend $f_G$ to $G^+$ independent of the determinant.  Let $f_D(x) \in C_c^\infty(\bR^+)$.  Define 
\begin{equation}
 f_{\pm}(g\cdot x_{\pm}) = f_G(g)f_D(\chi(g)).
\end{equation}
The twisted orbital integrals are given by
\begin{align}
 Z^{\pm}(f_{\pm}, \E_r, L; s) &= \int_{G^+/\Gamma} \chi(g)^s \E_r(g^{-1}) \sum_{x \in L} f_{\pm}(g\cdot x) dg.
 \end{align}

 \begin{lemma}\label{factorization_lemma}
  In $\RE(s)>1$, 
  \begin{equation}
   Z^{\pm}(f_{\pm},\E_r, L;s) = \frac{\sqrt{\pi}K_{\frac{z}{2}}(2)}{12} \sL^{\pm}(\E_r,s)\tilde{f}_D(s).
  \end{equation}

 \end{lemma}
\begin{proof}
 Calculate
 \begin{align*}
  Z^+(f_+, \E_r, L; s) &= \int_{G^+/\Gamma} \chi(g)^s \E_r(g^{-1}) \sum_{x \in L} f_+(g\cdot x) dg\\
  &=\int_{G^+/\Gamma} \chi(g)^s \E_r(g^{-1}) \sum_{m=1}^\infty \sum_{i=1}^{h(m)} \frac{1}{|\Gamma(i,m)|} \sum_{\gamma \in \Gamma} f_+(g\gamma g_{i,m} \cdot x_+) dg\\
  &= \int_{G^+} \chi(g)^s \E_r(g^{-1}) \sum_{m=1}^\infty \sum_{i=1}^{h(m)} \frac{1}{|\Gamma(i,m)|} f_G(g g_{i,m} ) f_D(\chi(g) m)\\
  &= \frac{\tilde{f}_D(s)}{12} \sum_{m=1}^\infty \frac{1}{m^s} \sum_{i=1}^{h(m)} \frac{\E_r(g_{i,m})}{|\Gamma(i,m)|}\int_{G^1} f_G(g) t(g)^{1+z} dg\\
  &= \sqrt{\pi} K_{\frac{z}{2}}(2) \frac{\tilde{f}_D(s)}{12} \sL^+(\E_r, s).
 \end{align*}
The proof in the case $Z^-$ is similar.
\end{proof}

 Define
 \begin{align}
  Z^{\pm, +}(f_{\pm}, \E_r, L; s) &= \int_{G^+/\Gamma, \chi(g) \geq 1} \chi(g)^s \E_r(g^{-1}) \sum_{x \in L}  f_{\pm}(g\cdot x) dg\\
 \notag \hat{Z}^{\pm, +}(\hat{f}_{\pm}, \E_r, \hat{L}; 1-s) &= \int_{G^+/\Gamma, \chi(g) \geq 1} \chi(g)^{1-s} \E_r(g^{-1}) \sum_{x \in \hat{L} \setminus \hat{L}_0}  \hat{f}_{\pm}\left(g \cdot x\right) dg\\
 \notag Z^{\pm, 0}(\hat{f}_{\pm}, \E_r, \hat{L}; s) &= \int_{G^+/\Gamma, \chi(g) \leq 1} \chi(g)^{s-1} \E_{r}(g^{-1}) \sum_{x \in \hat{L}_0}  \hat{f}_{\pm}\left(g^\iota \cdot x\right) dg.
\end{align}
As before, the first two integrals are entire, due to the rapid decay of $f$ and $\hat{f}$.
The last integral is equal to 
\begin{align*}
 Z^{\pm, 0}\left(\hat{f}_{\pm}, \E_r, \hat{L}; s \right) &= \int_{G^+/\Gamma, \chi(g) \leq 1} \chi(g)^{s-1} \E_r(g^{-1}) \sum_{x \in \hat{L}_0} \hat{f}_{\pm}(g^\iota \cdot x) dg\\
 &= \int_0^1\frac{d\lambda}{\lambda} \lambda^{12s-12}\int_{G^1/\Gamma} \E_{r}(g^{-1}) \sum_{x \in \hat{L}_0} \hat{f}_{\pm}^{\lambda^{-3}}(g \cdot x) dg.
\end{align*}

 The orbital integral satisfies a split functional equation, which is a result of applying the Poisson summation formula.
\begin{lemma}\label{split_functional_equation}
 We have
 \begin{equation}
  Z^{\pm}(f_{\pm}, \E_r, L;s) = Z^{\pm, +}(f_{\pm}, \E_r, L; s) + \hat{Z}^{\pm, +}(\hat{f}_{\pm}, \E_r, \hat{L}; 1-s) + Z^{\pm, 0}(\hat{f}_{\pm}, \E_r, \hat{L}; s).
 \end{equation}

\end{lemma}
The main proposition to be proved in this section is as follows.

Write $f \sim g$ if $f-g$ is entire.
\begin{proposition}\label{pole_proposition}
 In the case that $f$ is supported on $V_-$, 
 \begin{align*}
  &Z^{-, 0}(f_-, \E_r, L;s) \sim \\&\sqrt{\pi}K_{\frac{z}{2}}(2) \Biggl[\frac{\zeta\left(\frac{1-z}{3} \right)2^{\frac{z-1}{6}}3^{-1}\pi^{\frac{1+2z}{6}}}{12s-11-z} \cos\left(\frac{\pi(1-z)}{6} \right)\frac{\Gamma\left(\frac{1-z}{3} \right)\Gamma\left(\frac{4-z}{6}\right)}{\Gamma\left(\frac{7-z}{6}\right)}\tilde{f}_D\left(\frac{11+z}{12}\right) \\& + \frac{\xi(z)}{\xi(1+z)}  
  \frac{\zeta\left(\frac{1+z}{3} \right)2^{\frac{-z-1}{6}}3^{-1}\pi^{\frac{1-2z}{6}}}{12s-11+z} \cos\left(\frac{\pi(1+z)}{6} \right)\frac{\Gamma\left(\frac{1+z}{3} \right)\Gamma\left(\frac{4+z}{6}\right)}{\Gamma\left(\frac{7+z}{6}\right)}\tilde{f}_D\left(\frac{11-z}{12}\right)\\
  &+\frac{\zeta(3+z)2^{\frac{-5-z}{2}}}{12s - 15-3z}\tilde{f}_D\left(\frac{5+z}{4} \right) +\frac{\xi(z)}{\xi(1+z)}\frac{\zeta(3-z)2^{\frac{-5+z}{2}}}{12s - 15+3z}\tilde{f}_D\left(\frac{5-z}{4} \right)\Biggr].
 \end{align*}
In the case that $f$ is supported on $V_+$,
\begin{align*}
 & Z^{+,0}(f_+, \E_r, L;s) \sim \\ &\sqrt{\pi}K_{\frac{z}{2}}(2) \Biggl[\frac{\zeta\left(\frac{1-z}{3} \right)3^{\frac{z-7}{4}} 2^{\frac{z-1}{6}}\pi^{\frac{1+2z}{6}}}{12s-11-z} \cos\left(\frac{\pi(1-z)}{6} \right)\frac{\Gamma\left(\frac{1-z}{3} \right)\Gamma\left(\frac{4-z}{6}\right)}{\Gamma\left(\frac{7-z}{6}\right)}\tilde{f}_D\left(\frac{11+z}{12}\right) \\& + \frac{\xi(z)}{\xi(1+z)}  
  \frac{\zeta\left(\frac{1+z}{3} \right)3^{\frac{-z-7}{4}}2^{\frac{-z-1}{6}}\pi^{\frac{1-2z}{6}}}{12s-11+z} \cos\left(\frac{\pi(1+z)}{6} \right)\frac{\Gamma\left(\frac{1+z}{3} \right)\Gamma\left(\frac{4+z}{6}\right)}{\Gamma\left(\frac{7+z}{6}\right)}\tilde{f}_D\left(\frac{11-z}{12}\right)\\
  &+\frac{\zeta(3+z)2^{\frac{-5-z}{2}}3^{\frac{1+z}{4}}}{12s - 15-3z}\tilde{f}_D\left(\frac{5+z}{4} \right) +\frac{\xi(z)}{\xi(1+z)}\frac{\zeta(3-z)2^{\frac{-5+z}{2}}3^{\frac{1-z}{4}}}{12s - 15+3z}\tilde{f}_D\left(\frac{5-z}{4} \right)\Biggr].
 \end{align*}
\end{proposition}
\begin{proof} 
 This is a result of combining Lemmas \ref{Theta_n_lemma}, \ref{Theta_1_c_lemma}, and \ref{Theta_2_c_lemma} below.
\end{proof}

Combined with the factorization formula in Lemma \ref{factorization_lemma}, and the split functional equation in Lemma \ref{split_functional_equation}, this proves Theorem \ref{main_theorem}.

\subsection{The singular integral}
As in \cite{H19}, set 
\begin{equation}
 J\left(\hat{f}\right)(g) = \sum_{x \in \hat{L}_0}  \hat{f}\left(g\cdot x\right).
\end{equation}

The following lemma is proved in \cite{H19}.
\begin{lemma}
 Suppose for some $A>4$ that $\hat{f}(x) \ll \frac{1}{(1+ \|x\|)^A}$.  Then $J\left(\hat{f}\right) \in C(G^1/\Gamma, A-6)$.
\end{lemma}

The object of interest is 
\begin{equation}
 \sI\left(\hat{f}, \E_r\right) = \int_{G^1/\Gamma} \E_r\left(g^{-1}\right) J\left(\hat{f}\right)(g) dg
\end{equation}
since
\begin{equation}
 Z^{\pm, 0}\left(\hat{f}_{\pm}, \E_r, \hat{L};s \right) = \int_0^1 \lambda^{12s-12}\sI\left(\hat{f}_{\pm}^{\lambda^{-3}} ,\E_r\right) \frac{d\lambda}{\lambda}.
\end{equation}

In order to gain convergence in later integrals, we reinterpret this as a limit of an integral against an incomplete Eisenstein series as in \cite{S72},
\begin{equation}
 \sI\left(\hat{f}, \E_r\right) = \frac{\xi(2)}{\psi(1)}\lim_{w \to 1^+} (w-1) \int_{G^1/\Gamma} \sE(\psi, w; g) \E_r\left(g^{-1}\right) J\left(\hat{f}\right)(g) dg.
\end{equation}

Write the singular forms in $\hat{L}_0$ as
\begin{align}
 \hat{L}_0 &= \{0\} \sqcup \bigsqcup_{m=1}^\infty \bigsqcup_{\gamma \in \Gamma/\Gamma\cap N} \{\gamma \cdot (0,0,0,m)\} \sqcup \bigsqcup_{m=1}^\infty \bigsqcup_{n\in \zed} \bigsqcup_{\gamma \in \Gamma/\Gamma\cap N} \{\gamma \cdot (0,0,3m,n)\}\\
 \notag &= \{0\} \sqcup L_0(I) \sqcup \hat{L}_0(II).
\end{align}
Write 
\begin{equation}
 \sI\left(\hat{f}, \E_r\right) = \Theta^{(0)}(\E_r) + \Theta^{(1)}(\E_r) + \Theta^{(2)}(\E_r)
\end{equation}
as the sum of three limits.   
Since $\E_r$ is mean 0 on $\Gamma \backslash G^1/K$, $\Theta^{(0)} = 0$.  In the remaining two pieces it is necessary to separate the contributions of the constant term of the Eisenstein series and the non-constant terms.

Write
\begin{align}
 \Theta^{(1)}(\E_r) &= \frac{\xi(2)}{\psi(1)}\lim_{w \to 1^+} (w-1)\int_{G^1/\Gamma}\sE(\psi, w; g) \E_r(g^t) \sum_{x \in L_0(I)}  \hat{f}\left(g\cdot x \right)dg\\
 \notag &= \frac{\xi(2)}{\psi(1)}\lim_{w \to 1^+} (w-1)\\ \notag &\times\int_{G^1/\Gamma \cap N}\sE(\psi, w; g) \E_r(g^t) \sum_{m=1}^\infty  \hat{f}\left(g \cdot \left(0,0,0, m\right)\right) dg.
\end{align}
Define
\begin{align}
 \Theta^{(1),c}(\E_r)&= \frac{\xi(2)}{\psi(1)}\lim_{w \to 1^+} (w-1)\\ \notag &\times\int_{G^1/\Gamma \cap N} \sE(\psi, w; g)\E_r^c(g^t) \sum_{m=1}^\infty  \hat{f}\left(g \cdot \left(0,0,0, m\right)\right) dg\\
 \notag \Theta^{(1),n}(\E_r)&= \frac{\xi(2)}{\psi(1)}\lim_{w \to 1^+} (w-1)\\ \notag &\times\int_{G^1/\Gamma \cap N} \sE(\psi, w; g)\E_r^n(g^t) \sum_{m=1}^\infty  \hat{f}\left(g \cdot \left(0,0,0, m\right)\right) dg.
\end{align}
Similarly,
\begin{align}
 \Theta^{(2)}(\E_r) &= \frac{\xi(2)}{\psi(1)}\lim_{w \to 1^+} (w-1)\\ \notag &\times\int_{G^1/\Gamma}\sE(\psi, w; g) \E_r(g^t) \sum_{x \in \hat{L}_0(II)}  \hat{f}\left(g\cdot x \right)dg\\
 \notag &= \frac{\xi(2)}{\psi(1)}\lim_{w \to 1^+} (w-1)\\ \notag &\times\int_{G^1/\Gamma \cap N}\sE(\psi, w; g) \E_r(g^t) \sum_{m=1}^\infty \sum_{n\in \zed} \hat{f}\left(g \cdot \left(0,0,3m, n\right)\right) dg
\end{align}
and
\begin{align}
 \Theta^{(2),c}(\E_r)&= \frac{\xi(2)}{\psi(1)}\lim_{w \to 1^+} (w-1)\\ \notag &\times\int_{G^1/\Gamma \cap N} \sE(\psi, w; g)\E_r^c(g^t) \sum_{m=1}^\infty \sum_{n\in\zed} \hat{f}\left(g \cdot \left(0,0,3m, n\right)\right) dg\\
 \notag \Theta^{(2),n}(\E_r)&= \frac{\xi(2)}{\psi(1)}\lim_{w \to 1^+} (w-1)\\ \notag &\times\int_{G^1/\Gamma \cap N}\sE(\psi, w; g) \E_r^n(g^t) \sum_{m=1}^\infty \sum_{n \in \zed} \hat{f}\left(g \cdot \left(0,0,3m, n\right)\right) dg.
\end{align}

Due to the exponential decay of $\E_r^n(g^t$) in the cusp, we may in fact obtain 
\begin{align}
 \Theta^{(1),n}(\E_r) &= \int_{G^1/\Gamma \cap N} \E_r^n(g^t) \sum_{m=1}^\infty  \hat{f}\left(g \cdot \left(0,0,0, m\right)\right) dg\\ \notag
 \Theta^{(2),n}(\E_r)&=\int_{G^1/\Gamma \cap N} \E_r^n(g^t) \sum_{m=1}^\infty \sum_{n \in \zed} \hat{f}\left(g \cdot \left(0,0,3m, n\right)\right) dg.
\end{align}

\subsection{The non-constant term}

\begin{lemma}
 We have $\Theta^{(1),n}(\E_r) = 0$.
\end{lemma}
\begin{proof}
 This follows as in the proof of Lemma 12 of \cite{H19}.  
\end{proof}

Define 
\begin{equation}\hat{G}_\lambda(x) =\sum_{\ell, m=1}^\infty \frac{\eta_{\frac{z}{2}}(3\ell m)}{\ell^{1+x}(3m)^{1 + 3x}}.
\end{equation}

\begin{lemma}
  For $\epsilon>0$, the function $\hat{G}_\lambda(x)$ is bounded on $\{x: \RE(x) \geq \epsilon\}$ by a constant depending only upon $\epsilon$.
\end{lemma}
\begin{proof}
This follows since the divisor function $\eta_{\frac{z}{2}}(n)$ grows slower than any power of $n$.
\end{proof}

Let 
\begin{align}
 W_\lambda(w_1, w_2) &= \frac{1}{2}\frac{\Gamma(1-w_2)\cos\left(\frac{\pi}{2}(1-w_2) \right)}{(2\pi)^{\frac{1 + w_1 + w_2}{2}}} \tilde{K}_{\frac{z}{2}}\left(\frac{w_1 + 3w_2-1}{2} \right);\\
 \notag & \tilde{K}_\nu(s) = 2^{s-2}\Gamma\left(\frac{s+\nu}{2} \right)\Gamma\left(\frac{s-\nu}{2} \right).
\end{align}

\begin{lemma}
 $W_\lambda$ is holomorphic in $\RE(w_1 + 3w_2)>1$, $\RE(w_2)<1$.  Let $0 < \epsilon < \frac{1}{2}$.  For $\epsilon \leq \RE(w_2) \leq 1-\epsilon$,
 \begin{equation}
  \left|\Gamma(1-w_2)\cos\left(\frac{\pi}{2}(1-w_2)\right)\right| \ll |w_2|^{\frac{1}{2}- \RE(w_2)}.
 \end{equation}

\end{lemma}

\begin{proof}
 See \cite{H19}, Lemma 14.
\end{proof}

Define $\Sigma_1^{\pm}(f, z_1, z_2) = \int_0^\infty \int_0^\infty f(0,0,t, \pm u) t^{z_1-1}u^{z_2 - 1} dt du$.
\begin{lemma}
 If $f$ is Schwarz class, then $\Sigma_1^{\pm}(f, z_1, z_2)$ is holomorphic in $\RE(z_1), \RE(z_2) >0$.  In this domain, for $\sigma_1, \sigma_2 > 0$,
 \begin{equation}
  \left|\Sigma_1^{\pm}(f, \sigma_1 + it_1, \sigma_2 + it_2)\right| \ll_{\sigma_1, \sigma_2, A_1, A_2} \frac{1}{(1+|t_1|)^{A_1}(1 + |t_2|)^{A_2}}.
 \end{equation}
For $t>0$, if $f^t(x) = f(tx)$ then $\Sigma_1^{\pm}(f^t, z_1, z_2) = t^{-z_1 - z_2}\Sigma_1^{\pm}(f, z_1, z_2).$
\end{lemma}

\begin{proof}
 See \cite{H19}, Lemma 15.
\end{proof}

\begin{lemma}
We have
 \begin{align}
  \hat{\Theta}^{(2),n}(\E_r) &= \frac{4}{\xi(z+1)} \sum_{\epsilon = \pm} \oiint_{\substack{\RE(w_1, w_2)\\ = (1, \frac{1}{2})}}  \Sigma_1^\epsilon\left(\hat{f}, w_1, w_2\right) W_\lambda(w_1, w_2) \hat{G}_{\lambda}\left(\frac{w_1 + w_2 - 1}{2} \right)dw_1 dw_2.
 \end{align}

\end{lemma}
\begin{proof}
 See \cite{H19} Lemma 16, where the details are the same in the case of an even form.
\end{proof}

\begin{lemma}\label{Theta_n_lemma}
 The contribution to $Z^{\pm, 0}\left(\hat{f}, \E_r, \hat{L}; s\right)$ from $\Theta^{(*), n}(\E_r)$ is entire.
\end{lemma}

\begin{proof}
 The contribution is
 \begin{align*}
  &\frac{4}{\xi(z+1)} \int_0^1 \frac{d\lambda}{\lambda}\lambda^{12s-12}\sum_{\epsilon= \pm} \oiint_{\substack{\RE(w_1, w_2)\\=(1, \frac{1}{2})}} \Sigma_1^\epsilon\left(\hat{f}^{\lambda^{-3}}, w_1, w_2\right) W_\lambda(w_1, w_2) \hat{G}_\lambda\left(\frac{w_1+w_2-1}{2} \right)dw_1 dw_2 \\
  &= \frac{4}{\xi(z+1)} \int_0^1 \frac{d\lambda}{\lambda}\lambda^{12s-12}\sum_{\epsilon= \pm} \\&\times\oiint_{\substack{\RE(w_1, w_2)\\=(1, \frac{1}{2})}}\lambda^{3(w_1 + w_2)} \Sigma_1^\epsilon\left(\hat{f}, w_1, w_2\right) W_\lambda(w_1, w_2) \hat{G}_\lambda\left(\frac{w_1+w_2-1}{2} \right)dw_1 dw_2.
 \end{align*}
The $w_1$ contour may be pushed arbitrarily far to the right, which proves that the function is entire in $s$.
\end{proof}

\subsection{The constant term} In this section we obtain the residues of the twisted zeta functions which arise from the constant term by modifying the method of Shintani \cite{S72}.

\begin{lemma}
 We have
 \begin{align}
 & \Theta^{(1),c}(\E_r) = \frac{1}{3}\left[\zeta\left(\frac{1-z}{3} \right)\Sigma_2\left(\hat{f}, \frac{1-z}{3}\right) + \frac{\xi(z)}{\xi(z+1)}\zeta\left(\frac{1+z}{3} \right)\Sigma_2\left(\hat{f}, \frac{1+z}{3}\right)  \right].
 \end{align}

\end{lemma}
\begin{proof}
 We have 
 \begin{align}
  &\Theta^{(1),c}(\E_r) = \frac{\xi(2)}{\psi(1)} \lim_{w \to 1^+}(w-1)\\
  \notag &\times \int_{G^1/\Gamma\cap N}\sE(\psi,w;g)\left(t^{1 + z} + \frac{\xi(z)}{\xi(1+z)} t^{1-z}\right)  \hat{f}\left(g \cdot \left(0,0,0, m\right)\right)dg.
 \end{align}
Since $\hat{f}$ is invariant under $k_\theta$ on the left and $\left(0,0,0, m\right)$ is invariant under $n_u$, in the $KAN$ decomposition, integration over $K$ may be eliminated, while integration over $n_u$ selects the constant term from $\sE(\psi, w;g)$. For $1<x_0 < w$ write the integral as  
\begin{align}
& \int_0^\infty \left(\oint_{\RE u = x_0} \frac{t^{1+u} + \frac{\xi(u)}{\xi(u+1)} t^{1-u}}{w-u}\psi(u)\right)\\ \notag &\times \left(t^{1+z} + \frac{\xi(z)}{\xi(1+z)}t^{1-z}\right) \sum_{m=1}^\infty  \hat{f}\left(0,0,0,\frac{m}{ t^3} \right) \frac{dt}{t^3}.
\end{align}
This expresses the integral as
\begin{align}
 &\sum_{m=1}^\infty \int_0^\infty \left(\oint_{\RE u = x_0} \frac{\psi(u)}{w-u} \left[t^{u+z} + \frac{\xi(u) t^{-u+z} }{\xi(u+1)}+ \frac{\xi(z)t^{u-z}}{\xi(z+1)} + \frac{\xi(u)\xi(z)t^{-u-z}}{\xi(u+1)\xi(z+1)} \right]\right)\\\notag &\times\hat{f}\left(0,0,0, \frac{m}{ t^3} \right)\frac{dt}{t}\\
 \notag &=\frac{1}{3} \Biggl\{  \oint_{\substack{\RE u = x_1\\ x_1 < -3}} \frac{\psi(u)}{w-u} \\&\times \notag\left(\zeta\left(-\frac{u+z}{3} \right)\Sigma_2\left(\hat{f}, - \frac{u+z}{3}\right) + \frac{\xi(z)}{\xi(z+1)}\zeta\left(-\frac{u-z}{3} \right)\Sigma_2\left(\hat{f}, -\frac{u-z}{3}\right) \right)du\\
 \notag & +\oint_{\substack{\RE u = x_2\\ x_2 > 3}} \frac{\psi(u)}{w-u}\frac{\xi(u)}{\xi(u+1)}\\&\times\notag \left(\zeta\left(\frac{u-z}{3} \right)\Sigma_2\left(\hat{f}, \frac{u-z}{3}\right) + \frac{\xi(z)}{\xi(z+1)}\zeta\left(\frac{u+z}{3} \right)\Sigma_2\left(\hat{f}, \frac{u+z}{3}\right) \right)du\Biggr\}.
\end{align}
Only the second integral contributes to the limit as $w\to 1^+$ since the first is holomorphic in $w$ there.  Picking up the pole at $u=1$ in the second integral obtains the claim.

\end{proof}

\begin{lemma}\label{Theta_1_c_lemma}
 When $f$ is supported on $V_-$ the contribution to $Z^{\pm, 0}\left(\hat{f}, \E_r, \hat{L};s\right)$ from $\Theta^{(1),c}(\E_r)$ is
 \begin{align*}
&\sqrt{\pi}K_{\frac{z}{2}}(2)\Biggl[\frac{\zeta\left(\frac{1-z}{3} \right)2^{\frac{z-1}{6}}3^{-1}\pi^{\frac{1+2z}{6}}}{12s-11-z} \cos\left(\frac{\pi(1-z)}{6} \right)\frac{\Gamma\left(\frac{1-z}{3} \right)\Gamma\left(\frac{4-z}{6}\right)}{\Gamma\left(\frac{7-z}{6}\right)}\tilde{f}_D\left(\frac{11+z}{12}\right) \\& + \frac{\xi(z)}{\xi(1+z)}  
  \frac{\zeta\left(\frac{1+z}{3} \right)2^{\frac{-z-1}{6}}3^{-1}\pi^{\frac{1-2z}{6}}}{12s-11+z} \cos\left(\frac{\pi(1+z)}{6} \right)\frac{\Gamma\left(\frac{1+z}{3} \right)\Gamma\left(\frac{4+z}{6}\right)}{\Gamma\left(\frac{7+z}{6}\right)}\tilde{f}_D\left(\frac{11-z}{12}\right) \Biggr].
 \end{align*}
When $f$ is supported on $V_+$, the contribution is
\begin{align*}
&\sqrt{\pi}K_{\frac{z}{2}}(2)\Biggl[\frac{\zeta\left(\frac{1-z}{3} \right) 2^{\frac{z-1}{6}}3^{\frac{z-7}{4}}\pi^{\frac{1+2z}{6}}}{12s-11-z} \cos\left(\frac{\pi(1-z)}{6} \right)\frac{\Gamma\left(\frac{1-z}{3} \right)\Gamma\left(\frac{4-z}{6}\right)}{\Gamma\left(\frac{7-z}{6}\right)}\tilde{f}_D\left(\frac{11+z}{12}\right) \\& + \frac{\xi(z)}{\xi(1+z)}  
  \frac{\zeta\left(\frac{1+z}{3} \right)2^{\frac{-z-1}{6}}3^{\frac{-z-7}{4}}\pi^{\frac{1-2z}{6}}}{12s-11+z} \cos\left(\frac{\pi(1+z)}{6} \right)\frac{\Gamma\left(\frac{1+z}{3} \right)\Gamma\left(\frac{4+z}{6}\right)}{\Gamma\left(\frac{7+z}{6}\right)}\tilde{f}_D\left(\frac{11-z}{12}\right) \Biggr].
\end{align*}

\end{lemma}

\begin{proof}
We show only the negative discriminant case, since the positive is multiplied by a factor $3^{\frac{3\left(\frac{1\pm z}{3}\right)}{4}-1} = 3^{\frac{\pm z-3}{4}}$, see Lemma \ref{Sigma_2_lemma}.  The negative discriminant case is 
 \begin{align*}
  &\int_0^1 \frac{d \lambda}{\lambda}\lambda^{12s-12} \frac{1}{3}\left[\zeta\left(\frac{1-z}{3}\right)\Sigma_2\left(\hat{f}^{\lambda^{-3}}, \frac{1-z}{3}\right) + \frac{\xi(z)}{\xi(1+z)}\zeta\left(\frac{1 + z}{3} \right)\Sigma_2\left(\hat{f}^{\lambda^{-3}}, \frac{1+z}{3} \right) \right]\\
  &= \frac{1}{3} \left[\frac{\zeta\left(\frac{1-z}{3} \right)\Sigma_2\left(\hat{f}, \frac{1-z}{3} \right)}{12s -11-z}  +\frac{\xi(z)}{\xi(1+z)}\frac{\zeta\left(\frac{1+z}{3} \right)\Sigma_2\left(\hat{f}, \frac{1+z}{3} \right)}{12s-11+z}   \right]\\
  &=\frac{1}{3}\Biggl[\frac{\zeta\left(\frac{1-z}{3} \right)2^{\frac{z-1}{6}}\pi^{\frac{2+z}{3}}}{12s-11-z} \cos\left(\frac{\pi(1-z)}{6} \right)\frac{\Gamma\left(\frac{1-z}{3} \right)\Gamma\left(\frac{4-z}{6}\right)}{\Gamma\left(\frac{7-z}{6}\right)}\tilde{f}_D\left(\frac{11+z}{12}\right) K_{\frac{z}{2}}(2)\\& + \frac{\xi(z)}{\xi(1+z)}  
  \frac{\zeta\left(\frac{1+z}{3} \right)2^{\frac{-z-1}{6}}\pi^{\frac{2-z}{3}}}{12s-11+z} \cos\left(\frac{\pi(1+z)}{6} \right)\frac{\Gamma\left(\frac{1+z}{3} \right)\Gamma\left(\frac{4+z}{6}\right)}{\Gamma\left(\frac{7+z}{6}\right)}\tilde{f}_D\left(\frac{11-z}{12}\right) K_{\frac{z}{2}}(2)\Biggr].
 \end{align*}

\end{proof}

This contributes the first set of poles in Proposition \ref{pole_proposition}.

\begin{lemma}
 We have
 \begin{align}
 &\Theta^{(2),c}(\E_r) \\\notag&= \frac{1}{2}\left[\zeta(3+z) \int_{x = (x_2,x_3,x_4)}f_0(x) |x_2|^{2+z} + \frac{\xi(z)}{\xi(1+z)}\zeta(3-z)\int_{x=(x_2, x_3, x_4)} f_0(x)|x_2|^{2-z} \right].
\end{align}

\end{lemma}

\begin{proof}
 Calculate
 \begin{align}
 &\notag \Theta^{(2),c}(\E_r) = \frac{\xi(2)}{\psi(1)} \lim_{w \to 1^+} (w-1) \int_{G^1/\Gamma\cap N} \sE(\psi, w;g)\\&\notag\times\left[t^{1+z} + \frac{\xi(z)}{\xi(z+1)}t^{1-z} \right]\sum_{m=1}^\infty \sum_{n \in \zed}  \hat{f}\left(g \cdot \left(0,0, 3m, n\right)\right)dg
 \end{align}
The integral may be written
\begin{align}
 \int_{G^1/\Gamma \cap N} \sE(\psi, w;g)\notag\left[t^{1+z} + \frac{\xi(z)}{\xi(z+1)}t^{1-z}\right] \sum_{m=1}^\infty \sum_{n \in \zed} \hat{f}\left(g \cdot \left(0,0,3m, n \right)\right)dg.
\end{align}
Using that $\hat{f}$ is left $K$ invariant, write this as 
\begin{align}
&  \int_0^\infty \left[t^{-2 + z} + \frac{\xi(z)}{\xi(z+1)}t^{-2-z} \right]dt \int_0^1 du \\\notag &\times \sum_{m=1}^\infty \sum_{n \in \zed} \hat{f}\left(a_t \cdot \left(0,0, 3m, 3m u + n\right)\right) \oint_{\substack{\RE v = x_0\\ 4 < x_0 < \RE w}}\frac{E(v, a_t n_u) \psi(v)}{w-v} dv\\
 \notag &= 
 \int_0^\infty \left[t^{-2 + z} + \frac{\xi(z)}{\xi(z+1)}t^{-2-z} \right]dt \int_{-\infty}^\infty du \\\notag &\times \sum_{m=1}^\infty  \hat{f}\left(a_t \cdot \left(0,0, 3m, u\right)\right)\frac{1}{3m}\sum_{0 \leq n <3m} \oint_{\substack{\RE v = x_0\\ 4 < x_0 < \RE w}}\frac{E(v, a_t n_{\frac{u-n}{3m}}) \psi(v)}{w-v} dv.
\end{align}
The sum over $n$ selects the Fourier coefficients of $E(v, \cdot)$ which are divisible by $3m$.  Split the remaining coefficients into the constant term and the non-constant term.  The constant term contributes
\begin{align}
 & \int_0^\infty \left[t^{-2 + z} + \frac{\xi(z)}{\xi(z+1)}t^{-2-z}  \right]dt \int_{-\infty}^\infty du\\
 \notag &\times \sum_{m=1}^\infty \hat{f}\left(0,0, \frac{3m}{ t}, \frac{u}{ t^3} \right) \oint_{\substack{\RE v = x_0\\ 4 < x_0 < \RE w}} \frac{t^{1+v} + \frac{\xi(v)}{\xi(v+1)} t^{1-v}}{w-v} \psi(v) dv   .
\end{align}
Split the last contour integral into two parts corresponding to $t^{1+v}$ and $t^{1-v}$.  In the $t^{1+v}$ part, shift the integral leftward arbitrarily to show that this has no singularity at $w=1$.  This reduces to the $t^{1-v}$ part, which we write as
\begin{align}
 & \oint_{\substack{\RE v = x_0\\ 4 < x_0 < \RE w}} \frac{\xi(v)\psi(v)}{\xi(v+1)(w-v)} \int_0^\infty\left[t^{3 + z-v} +\frac{\xi(z)}{\xi(z+1)} t^{3 -z-v} \right]\frac{dt}{t}\\
 \notag &\times \int_{-\infty}^\infty du\sum_{m=1}^\infty \hat{f}\left(0,0,\frac{3m}{t},u \right)
\\
\notag & = \oint_{\substack{\RE v = x_0\\ 4 < x_0 < \RE w}} \frac{\xi(v)\psi(v)}{\xi(v+1)(w-v)} \\&\times \notag \biggl[\zeta(v-3-z) \left(\frac{1}{3} \right)^{v-3-z}\Sigma_3(\hat{f},v-3-z)\\&\notag + \frac{\xi(z)}{\xi(z+1)} \zeta(v-3+z)\left(\frac{1}{3} \right)^{v-3+z}\Sigma_3(\hat{f},v-3+z) \biggr].\end{align}
The contribution to $\Theta^{(2),c}(\E_r)$ comes from the pole at $v=1$ and obtains
\begin{align}
 &9\left[\zeta(-2-z)\left(\frac{1}{3} \right)^{-z}\Sigma_3(\hat{f}, -2-z) + \frac{\xi(z)}{\xi(z+1)} \zeta(-2 + z)\left(\frac{1}{3} \right)^z \Sigma_3(\hat{f}, -2 + z) \right].
\end{align}
Combine this with Lemma \ref{Sigma_3_lemma} and the functional equation of the Riemann zeta function to obtain the claimed quantity in the lemma,
\begin{equation}
\frac{1}{2}\left[\zeta(3+z) \int_{x = (x_2,x_3,x_4)}f_0(x) |x_2|^{2+z} + \frac{\xi(z)}{\xi(1+z)}\zeta(3-z)\int_{x=(x_2, x_3, x_4)} f_0(x)|x_2|^{2-z} \right].
\end{equation}

The contribution of the non-constant terms of the Fourier series for $E(v,\cdot)$ is
\begin{align}
 &\int_0^\infty dt \int_{-\infty}^\infty \oint_{\substack{\RE v = x_0\\ 4 < x_0 < w}} \frac{4}{\xi(v+1)} \frac{\psi(v)}{w-v}\\
 \notag &\times \left[t^{2+z} + \frac{\xi(z)}{\xi(z+1)} t^{2-z} \right] \sum_{\ell,m=1}^\infty \eta_{\frac{v}{2}} (3\ell m)K_{\frac{v}{2}}(6\pi \ell m  t^2) \cos(2\pi \ell t^3 u) \hat{f}\left(0,0,\frac{3m}{t}, u\right).
\end{align}
Since the Fourier transform $\hat{f}$ is Schwarz class, and the $K$ Bessel function has exponential decay in large variable, the $v$ integral may be passed to the left of the 1 line, which proves that this term is holomorphic at $w=1$, hence does not contribute to $\Theta^{(2),c}(\E_r).$
\end{proof}

\begin{lemma}\label{Theta_2_c_lemma}
 In the case that $f$ is supported on $V_-$, $\Theta^{(2),c}$ contributes  
 \begin{align*}
 \sqrt{\pi} K_{\frac{z}{2}}(2)\left[\frac{\zeta(3+z)2^{\frac{-5-z}{2}}}{12s - 15-3z}\tilde{f}_D\left(\frac{5+z}{4} \right) +\frac{\xi(z)}{\xi(1+z)}\frac{\zeta(3-z)2^{\frac{-5+z}{2}}}{12s - 15+3z}\tilde{f}_D\left(\frac{5-z}{4} \right)\right].
 \end{align*}
to $Z^{\pm, 0}\left(\hat{f}, \E_r, \hat{L};s\right)$.
 In the case that $f$ is supported on $V_+$, the contribution is
\begin{align*}
 \sqrt{\pi} K_{\frac{z}{2}}(2)\left[\frac{\zeta(3+z)2^{\frac{-5-z}{2}}3^{\frac{1+z}{4}}}{12s - 15-3z}\tilde{f}_D\left(\frac{5+z}{4} \right) +\frac{\xi(z)}{\xi(1+z)}\frac{\zeta(3-z)2^{\frac{-5+z}{2}}3^{\frac{1-z}{4}}}{12s - 15+3z}\tilde{f}_D\left(\frac{5-z}{4} \right)\right].
\end{align*}

\end{lemma}
\begin{proof}
Apply Lemma \ref{Sigma_3_lemma} to find
 \begin{align*}
 & \int_0^1 \frac{d\lambda}{\lambda}\lambda^{12s-12}\\&\times \Biggl[9\Biggl[\zeta(-2-z) 3^z \Sigma_3\left(\hat{f}^{\lambda^{-3}},-2-z\right)+ \frac{\xi(z)}{\xi(1+z)}\zeta(-2+z)3^{-z} \Sigma_3\left(\hat{f}^{\lambda^{-3}}, -2+z\right)\Biggr]\Biggr]\\
 &=\frac{\zeta(-2-z)3^{2+z}\Sigma_3\left(\hat{f}, -2-z\right)}{12s-15-3z} + \frac{\xi(z)}{\xi(1+z)}\frac{\zeta(-2+z)3^{2-z}\Sigma_3\left(\hat{f}, -2+z\right)}{12s-15+3z}\\
 &= \frac{1}{2}\left[\frac{\zeta(3+z)}{12s-15-3z}\int f_0(x)|x_2|^{2+z} +\frac{\xi(z)}{\xi(1+z)}\frac{\zeta(3-z)}{12s-15+3z}\int f_0(x)|x_2|^{2-z} \right].
 \end{align*}
In the case of $V_-$, by Lemma \ref{Phi_0_lemma} this is
\begin{align*}
 \sqrt{\pi} K_{\frac{z}{2}}(2)\left[\frac{\zeta(3+z)2^{\frac{-5-z}{2}}}{12s - 15-3z}\tilde{f}_D\left(\frac{5+z}{4} \right) +\frac{\xi(z)}{\xi(1+z)}\frac{\zeta(3-z)2^{\frac{-5+z}{2}}}{12s - 15+3z}\tilde{f}_D\left(\frac{5-z}{4} \right)\right].
\end{align*}
In the case of $V_+$, this is
\begin{align*}
 \sqrt{\pi} K_{\frac{z}{2}}(2)\left[\frac{\zeta(3+z)2^{\frac{-5-z}{2}}3^{\frac{1+z}{4}}}{12s - 15-3z}\tilde{f}_D\left(\frac{5+z}{4} \right) +\frac{\xi(z)}{\xi(1+z)}\frac{\zeta(3-z)2^{\frac{-5+z}{2}}3^{\frac{1-z}{4}}}{12s - 15+3z}\tilde{f}_D\left(\frac{5-z}{4} \right)\right].
\end{align*}
\end{proof}

\section{Reducible forms}

Introduce reducible versions of the zeta functions as follows, the superscript $r$ indicating sums are restricted to reducible forms,
\begin{equation}
 Z^{\pm, r}(f, \E_r, L;s) = \int_{G^+/\Gamma} \chi(g)^s \E_r(g^{-1}) {\sum_{x \in L }}^r  f(g\cdot x) dg
\end{equation}
and
\begin{align}
 \sL^{+, r}(\E_r, s) &= \sum_{m=1}^\infty \frac{1}{m^s}  {\sum_{i=1}^{h(m), r}}  \frac{\E_r(g_{i,m})}{|\Gamma(i,m)|},\\
 \notag
 \sL^{-,r}(\E_r, s) &= \sum_{m=1}^\infty \frac{1}{m^s} {\sum_{i=1}^{h(-m), r}}  \E_r(g_{i,-m}).
\end{align}
Recall that in the description of the region $\sR$ which is a fundamental domain for reducible forms modulo $\Gamma$ (see \cite{S75} pp.\ 45-46), the first coefficient $b$ is assumed to be positive.  Dropping this restriction and introducing a factor of $\frac{1}{2}$ to compensate, then using the $K$-invariance in the last line, \begin{align}
 Z^{\pm, r}(f_{\pm}, \E_r, L;s) &= \frac{1}{2}\int_{G^+/\Gamma} \chi(g)^s \E_r(g^t) \sum_{\gamma \in \Gamma/\Gamma\cap N}{\sum_{x =(0,b,c,d) \in \zed^4  }} \frac{f_{\pm}(g\gamma \cdot x)}{1 + 2 \one(c^2 -4bd = \square)}  dg \\
 &\notag = \frac{1}{2}\int_{G^+/\Gamma \cap N} \chi(g)^s \E_{r}(g^t) \sum_{x = (0,b,c,d)\in \zed^4} \frac{f_{\pm}(g \cdot x)}{1 + 2 \one(c^2 -4bd = \square)}dg\\
 &\notag = \frac{1}{2}\int_{B^+/B^+_\zed } \chi(g)^s \E_{r}(g^t) \sum_{x = (0,b,c,d)\in \zed^4} \frac{f_{\pm}(g \cdot x)}{1 + 2 \one(c^2 -4bd = \square)}dg.
\end{align}
Set 
\begin{align}
 Z_{1}^{\pm, r}(f_{\pm}, \E_r, L; s) &= \frac{1}{2}\int_{B^+/B_\zed^+}\chi(g)^s \E_r(g^t)\sum_{x = (0,b,c,d)}  f_{\pm}(g\cdot x)dg,\\ \notag
 Z_{\square}^{\pm, r}(f_{\pm}, \E_r, L;s) &= \frac{1}{3} \int_{B^+/B_{\zed}^+} \chi(g)^s \E_r(g^t) \sum_{\substack{x = (0,b, c, d) \in \zed^4\\ c^2 - 4bd = \square}} f_{\pm}(g \cdot x) dg.
\end{align}
Thus \[Z^{\pm, r}(f_{\pm}, \E_r, L; s) = Z_1^{\pm, r}(f_{\pm}, \E_r, L; s) - Z_{\square}^{\pm, r}(f_{\pm}, \E_r, L;s).\]

Indicate the action of $g \in \GL_2$ on binary cubic forms by $g_3 \cdot x$ and the action on binary quadratic forms by $g_2 \cdot x$. 
For $x = (b,c,d) \in \zed^3$, define $f_0(x) = f(0,x)$.  Write the integral as
\begin{align*}
 &Z_{1}^{\pm, r}(f_{\pm}, \E_r, L;s)\\ &= \frac{1}{2}\int_0^\infty \frac{d\lambda}{\lambda} \int_0^\infty \frac{dt}{t^3} \int_0^1 du \lambda^{12s} \E_{r}(n_u^{t}a_t) \sum_{x = (b,c,d)} f_{\pm,0}\left(\left(d_{\sqrt{\frac{\lambda^3}{t}}}a_t n_u\right)_2 \cdot x\right)\\
 &= \frac{1}{3}\int_0^\infty \frac{d\lambda}{\lambda} \int_0^\infty \frac{dt}{t^3} \int_0^1 du \lambda^{8s}t^{4s} \E_{r}(n_u^{t}a_t) \sum_{x = (b,c,d)}f_{\pm,0}\left(\left( d_{\lambda} a_t n_u\right)_2 \cdot x\right).
\end{align*}
Separate the Eisenstein series into its constant term and non-constant term.  The constant term part is, 
\begin{align*}
 Z_{1,c}^{\pm, r} &= \frac{1}{3} \int_0^\infty \frac{d\lambda}{\lambda} \int_0^\infty \frac{dt}{t^3} \int_0^1 du \lambda^{8s}\left(t^{4s+1+z} + \frac{\xi(z)}{\xi(1+z)}t^{4s+1-z} \right)\sum_{x = (b,c,d)} f_{\pm,0}((d_\lambda a_t n_u)_2 \cdot x). 
\end{align*}
Similarly,
\[
 Z_{1,n}^{\pm, r} = \frac{1}{3} \int_0^\infty \frac{d\lambda}{\lambda} \int_0^\infty \frac{dt}{t^3} \int_0^1 du \lambda^{8s}t^{4s}\E_r^n(n_u^{t}a_t)\sum_{x = (b,c,d)} f_{\pm,0}((d_\lambda a_t n_u)_2 \cdot x).
\]

Split the integral
\begin{align}
 Z_{1,*}^{\pm, +, r}(f, \E_r, L; s) = \frac{1}{3}\int_1^\infty \frac{d\lambda}{\lambda} \int_0^\infty \frac{dt}{t^3} \int_0^1 du \lambda^{12s}t^{4s} \E_{r}^*(n_u^{t}a_t) \sum_{x = (b,c,d)}  f_{\pm,0}((d_{\lambda} a_t n_u)_2 \cdot x).
\end{align}
In the part with $\lambda < 1$ perform Poisson summation.   Here we use the bilinear pairing $[x,\xi] = x_1 \xi_3 - \frac{1}{2}x_2 \xi_2 + x_3 \xi_1$.  Write
\begin{align}
\int_{\bR^3} f_{\pm,0}((d_{\lambda} a_t n_u)_2 \cdot x) e([x, \xi]) dx &= \frac{1}{\lambda^6} \int_{\bR^3} f_{\pm,0}(x) e([(d_{\lambda} a_t n_u)_2^{-1} \cdot x, \xi])\\
&\notag = \frac{1}{\lambda^6} \int_{\bR^3} f_{\pm,0}(x) e([x, (d_{\frac{1}{\lambda}} a_t n_u)_2 \cdot \xi]).
\end{align}

By Poisson summation, splitting the sum over dual forms into forms which are non-singular and singular.  Thus write the part of the integral with $\lambda < 1$ as
\begin{align}
 \notag \hat{Z}_{1,*}^{\pm, +, r}\left(\hat{f}, \E_r, \hat{L}; s\right) =& \frac{1}{3}\int_1^\infty \frac{d\lambda}{\lambda} \int_0^\infty \frac{dt}{t^3} \int_0^1 du  \lambda^{6-8s}t^{4s}  \E_r^*(n_u^{t}a_t)\\&\times\sum_{\xi \in \hat{L} \setminus \hat{L}_0}  \hat{f}_0\left(d_\lambda  a_t n_u \cdot \xi \right)\\ \notag
  \hat{Z}_{1,*}^{\pm, 0, r}\left(\hat{f}, \E_r, \hat{L}; s\right)=& \frac{1}{3}\int_0^1 \frac{d\lambda}{\lambda} \int_0^\infty \frac{dt}{t^3} \int_0^1 du\lambda^{8s-6}t^{4s}\E_r^*(n_u^{t}a_t)\\&\times \sum_{\xi \in  \hat{L}_0}  \hat{f}_0\left(d_{\frac{1}{\lambda}} a_t n_u \cdot \xi \right).
\end{align}

\begin{lemma}
 We have 
 \[
  Z_{1,*}^{\pm, r} = Z_{1,*}^{\pm, +, r} + \hat{Z}_{1,*}^{\pm, +, r} + \hat{Z}_{1,*}^{\pm, 0, r}.
 \]
\end{lemma}
\begin{proof}
 The decomposition follows from applying Poisson summation when $\lambda < 1$ and exchanging $\lambda$ with $\frac{1}{\lambda}$ in $\hat{Z}_{1,*}^{\pm, +, r}$.
 
% To check that $Z_{q,1}^{\pm, +, r}$ is entire, note that $f$ is supported on non-singular points so $b \neq 0$, and acting by $n_u$ cannot cause both $c$ and $d$ to vanish.  The discriminant of the binary cubic form scales like $\lambda^{12}$, while $a_t$ maps $(b,c,d)$ to $\left(tb, \frac{c}{t}, \frac{d}{t^3}\right)$.  It follows that the integrals over $\lambda$ and $t$ and the sum over $x = (b,c,d)$ are all finite, which proves the claim.

\end{proof}

\begin{lemma}
 The functions $Z_{1,*}^{\pm, +, r}, \hat{Z}_{1,*}^{\pm, +, r}$ are holomorphic in $\RE(s)> \frac{3}{4}$.  
\end{lemma}

\begin{proof}
 Write the integral in $Z_{1,*}^{\pm, +, r}$ as
 \begin{align*}
  &\frac{1}{3}\int_{1}^\infty \frac{d\lambda}{\lambda} \int_0^\infty \frac{dt}{t} \int_{-\infty}^\infty du \lambda^{8s} t^{4s-2}\E_{r}^*(n_u^{t}a_t) \sum_{\substack{x = (b,c,d)\\ -|b| \leq c < |b|}} f_0\left(\lambda^2\left(t^2 b, 2bu, \frac{d-\frac{c^2}{4b} + bu^2}{t^2}\right)\right)\\
  &=\frac{1}{3}\int_{1}^\infty \frac{d\lambda}{\lambda} \int_0^\infty \frac{dt}{t} \int_{-\infty}^\infty du \lambda^{8s} t^{4s-2}\E_{r}^*(n_u^{t}a_t) \sum_{\substack{x = (b,c,d)\\ -|b| \leq c < |b|}} \frac{1}{b} f_0\left(\lambda^2\left(t^2 b, 2u, \frac{\frac{4bd-c^2}{4} + u^2}{t^2b}\right)\right)\\
  &=\frac{1}{3}\int_{1}^\infty \frac{d\lambda}{\lambda} \int_0^\infty \frac{dt}{t} \int_{-\infty}^\infty du\\& \times\sum_{\substack{x = (b,c,d)\\ -|b| \leq c < |b|}} \frac{1}{b^{2s}} \lambda^{8s} t^{4s-2}\E_{r}^*\left(n_u^{t}a_{\frac{t}{\sqrt{|b|}}}^{-1}\right)  f_0\left(\lambda^2\left(t^2 , 2u, \frac{\frac{4bd-c^2}{4} + u^2}{t^2}\right)\right).
 \end{align*}
For large $t$, estimate using the rapid decay of $f_0$ in the first slot, to guarantee convergence.  For small $t$, bound 
\[
 \E_r\left(n_u^{-1} a_{\frac{t}{\sqrt{|b|}}}^{-1}\right) = \E_r\left(n_u^ta_{\frac{t}{\sqrt{|b|}}} \right)\ll \frac{\sqrt{|b|}}{t}.
\]
The discriminant of the form scales as $\lambda^4$, so the integral in $\lambda$ converges by rapid decay of $f_0$.  This also truncates to forms of bounded discriminant.  Bound the sum over $b$ by summing over each class of binary quadratic form up to $\SL_2(\zed)$ equivalence, then summing over the first coefficient. The number of first coefficients of size at most $x$ for a given form $f$ is bounded by $\ll x \log x$ as $x \to \infty$ by counting the number of points inside an ellipse in the negative discriminant case, and by counting the number of points under a hyperbola in the positive discriminant case. This obtains convergence if $s > \frac{3}{4}$.

The argument is the same in the case of $\hat{Z}$, except that the power on $\lambda$ is replaced with $6-8s$.
\end{proof}

\begin{lemma}
 When $f$ is supported on $V_-$, the constant term function $Z_{1,c}^{\pm, r}$ has a pair of simple poles at $\frac{5}{4}+\frac{z}{4}$ and $\frac{5}{4}-\frac{z}{4}$, with residues
 \begin{align*}
  \frac{1}{12} \zeta(3+ z)2^{\frac{-5- z}{2}} \tilde{f}_D\left(\frac{5+ z}{4} \right)\sqrt{\pi}K_{\frac{z}{2}}(2), \qquad &s = \frac{5}{4} + \frac{z}{4},\\
  \frac{\xi(z)}{\xi(1+z)}\frac{1}{12} \zeta(3- z)2^{\frac{-5+ z}{2}} \tilde{f}_D\left(\frac{5- z}{4} \right)\sqrt{\pi}K_{\frac{z}{2}}(2), \qquad &s = \frac{5}{4} - \frac{z}{4}.
 \end{align*}
 When $f$ is supported on $V_+$, the residues are 
  \begin{align*}
  \frac{1}{12} \zeta(3+ z)2^{\frac{-5- z}{2}}3^{\frac{1+ z}{4}} \tilde{f}_D\left(\frac{5+ z}{4} \right)\sqrt{\pi}K_{\frac{z}{2}}(2), \qquad &s = \frac{5}{4} + \frac{z}{4},\\
  \frac{\xi(z)}{\xi(1+z)}\frac{1}{12} \zeta(3- z)2^{\frac{-5+ z}{2}}3^{\frac{1- z}{4}} \tilde{f}_D\left(\frac{5- z}{4} \right)\sqrt{\pi}K_{\frac{z}{2}}(2), \qquad &s = \frac{5}{4} - \frac{z}{4}.
 \end{align*}
Besides these two poles, the function is holomorphic in $\RE(s)>\frac{5}{12}$.

\end{lemma}

\begin{proof}
 Recall that, with $Z(f, s_1, s_2)$ the orbital zeta function of the space of binary quadratic forms as in Shintani \cite{S75},
 \begin{align*}
   Z_{1,c}^{\pm, r} &= \frac{1}{3} \int_0^\infty \frac{d\lambda}{\lambda} \int_0^\infty \frac{dt}{t^3} \int_0^1 du \lambda^{8s}\left(t^{4s+1+z} + \frac{\xi(z)}{\xi(1+z)}t^{4s+1-z} \right)\sum_{x = (b,c,d)} f_{\pm,0}((d_\lambda a_t n_u)_2 \cdot x)\\
 &= \frac{1}{3}\left(Z\left(f_{\pm,0}, 2s+\frac{1}{2} + \frac{z}{2}, s - \frac{1}{4}-\frac{z}{4}\right) + \frac{\xi(z)}{\xi(1+z)}Z\left(f_{\pm,0}, 2s+\frac{1}{2}-\frac{z}{2}, s - \frac{1}{4} + \frac{z}{4}\right) \right). 
 \end{align*}
Recall 
\begin{align*}
 Z(f_{\pm,0}, s_1, s_2) &= Z_+(f_{\pm,0}, s_1, s_2) + Z_+^*\left(\hat{f}_{\pm,0}, s_1, \frac{3}{2}-s_1-s_2\right)\\
 & +\frac{1}{4} \frac{1}{2s_1 + 2s_2-3}\frac{\zeta(s_1) \zeta(2s_1-1)}{\zeta(2s_1)}\Sigma\left(\hat{f}_{\pm}, s_1-1\right)\\
 &+\frac{1}{8}\frac{\zeta(s_1)}{s_2-1}(\Phi_+ + \Phi_-)\left(f_{\pm,0}, s_1-1,0\right)\\
 &- \frac{1}{8} \frac{1}{s_2}\frac{\zeta(s_1)\zeta(2s_1-1)}{\zeta(2s_1)} \Sigma\left(f_{\pm,0}, s_1-1\right)\\
 &- \frac{1}{8} \frac{\zeta(s_1)}{2s_1+2s_2-1}(\Phi_+ + \Phi_-)\left(\hat{f}_{\pm,0}, s_1-1,0\right).
\end{align*}
Substituting $s_1 = 2s + \frac{1}{2} +\frac{z}{2}$, $s_2 = s-\frac{1}{4}-\frac{z}{4}$, all but the third line is holomorphic in $\RE(s)> \frac{3}{4}$.  The third line has a pole at $s = \frac{5+z}{4}$ where $s_2=1$ which obtains the claimed simple pole.  The second term is similar.

To obtain the claimed formulae, note that Lemma \ref{Phi_0_lemma} establishes that
\begin{align*}
 \Phi_+\left(f_{+,0}, 2\pm z, 0\right) &= 2^{\frac{-3\mp z}{2}}\tilde{f}_D\left(\frac{5\pm z}{4} \right)\sqrt{\pi} K_{\frac{z}{2}}(2),\\
 \Phi_-\left(f_{-,0}, 2\pm z, 0\right) &= 3^{\frac{1\pm z}{4}}2^{\frac{-3\mp z}{2}}\tilde{f}_D\left(\frac{5\pm z}{4} \right)\sqrt{\pi} K_{\frac{z}{2}}(2).
\end{align*}

\end{proof}
The poles and residues in this lemma at $s = \frac{5\pm z}{4}$ match those from Lemma \ref{Theta_2_c_lemma}.

\subsection{Singular forms}
The singular part of the zeta function $\hat{Z}_{1}^{\pm, 0, r}$ is treated using the decomposition of the singular set in Lemma \ref{reducible_singular_fibration_lemma}.  Taken together, the sum of the Fourier transform $\hat{f}_0$  is integrable, but it need not be over individual components.  We integrate against the function 
\[
 F_w(t)= \oint_{\RE \alpha = c} \psi(\alpha) \frac{t^{\alpha}}{\alpha} \frac{d\alpha}{w-\alpha}
\]
and take the pole at $w = 0$ to regularize the integrals.

The following lemma treats the case of the 0 form.
\begin{lemma}
In $\RE(s)> \frac{3}{4}$,
\[
 \lim_{w \downarrow 0} \frac{w}{\psi(0)} \frac{1}{3} \int_0^1 \frac{d\lambda}{\lambda} \lambda^{8s -6} \int_0^\infty \frac{dt}{t^3} t^{4s} \int_0^1 du \E_{r}^{n}(n_u^{t}a_t)F_w(t) =0.
\]

\end{lemma}

\begin{proof}
 Expand, via Fourier expansion with $\RE w > c>0$
 \begin{align*}
 & \frac{1}{3}\int_0^1 \frac{d\lambda}{\lambda} \lambda^{8s -6} \int_0^\infty \frac{dt}{t^3} t^{4s} \int_0^1 du \E_{r}^{n}(n_u^{t}a_t) F_w(t)\\
 &=\frac{1}{24s-18} \int_0^\infty \frac{dt}{t^3}t^{4s} \int_0^1 du\\ &\times \left( \frac{4t}{\xi(1+z)}\sum_{m_1=1}^\infty \eta_{\frac{z}{2}}(m_1) K_{\frac{z}{2}}(2\pi m_1 t^2) \cos(2\pi m_1 u) \right)\\
 &\times \oint_{\RE \alpha = c} \frac{\psi(\alpha)}{w-\alpha}\frac{t^{\alpha}}{\alpha}d\alpha =0
\end{align*}
by integrating in the parabolic ($u$) direction.
\end{proof}

\subsection{Case of $L_{0,r}(I)$}
The following lemma handles the residue from $L_{0,r}$.
\begin{lemma}
 We have
 \begin{align*}
  & \lim_{w \to 0}\frac{w}{\psi(0)} \frac{1}{3} \int_0^1 \frac{d\lambda}{\lambda}\lambda^{8s-6} \int_0^\infty \frac{dt}{t^3} t^{4s} \int_0^1 du F_w(t)\E_{r}^n(n_u^{t}a_t)  \sum_{\ell \neq 0} \hat{f}_0\left(0,0, \frac{\ell}{\lambda^3 t^2} \right)= 0.
 \end{align*}

\end{lemma}

\begin{proof}
 This follows by integrating in the parabolic direction.
\end{proof}

\subsection{Case of $\hat{L}_{0,r}(II)$}
\begin{lemma}
The non-constant term contribution of $\hat{L}_{0,r}(II)$ is holomorphic in $\RE(s)>\frac{1}{4}$.
\end{lemma}
\begin{proof}
Write $\hat{f}_{0,3}$ for the Fourier transform of $f_0$ in the first two coordinates but not the third. The term is given by
 \begin{align*}
  & \frac{1}{3}\int_0^1 \frac{d\lambda}{\lambda} \lambda^{8s - 6} \int_0^\infty \frac{dt}{t}t^{4s-1} F_w(t)\int_{-\infty}^\infty \frac{4}{\xi(1+z)} \sum_{m=1}^\infty \eta_{\frac{z}{2}}(m) K_{\frac{z}{2}}(2\pi mt^2) \cos(2\pi m u)\\
  &\times \sum_{k \neq 0} \sum_{0 \leq n < |2k|} \hat{f}\left(\lambda^{-2}\left(0, 2k, \frac{2ku + n}{t^2}\right)\right)\\
  &=\frac{4}{3\xi(1+z)}\int_0^1 \frac{d\lambda}{\lambda} \lambda^{8s -6} \int_0^\infty \frac{dt}{t}t^{4s-1}F_w(t)\int_{-\infty}^\infty \\ &\times \sum_{k \neq 0} \sum_{m=1}^{\infty}\eta_{\frac{z}{2}}(2|k|m) K_{\frac{z}{2}}(4\pi |k|mt^2) \cos(2\pi m u) \hat{f}_0\left(\lambda^{-2}\left(0, 2k, \frac{u}{t^2} \right) \right)\\
  &= \frac{4}{3\xi(1+z)}\int_0^1 \frac{d\lambda}{\lambda} \lambda^{8s-4} \int_0^\infty \frac{dt}{t}  t^{4s+1}F_w(t)\\
  &\times \sum_{k \neq 0}\sum_{m =1}^\infty \eta_{\frac{z}{2}}(2|k|m) K_{\frac{z}{2}}(4\pi |k|mt^2) \hat{f}_{0,3}\left(0, \frac{2k}{\lambda^2}, 2\lambda^2 t^2 m\right).  
 \end{align*}
The $K$ Bessel function essentially truncates summation at $|k|m \ll t^{-2}$, which suffices for convergence of the integral over $t$.  The integral over $\lambda$ converges due to the middle term $\frac{2k}{\lambda^2}$.
\end{proof}

\subsection{Case of $\hat{L}_{0,r}(III)$}
\begin{lemma}
 The contribution of the non-constant term part of $\hat{L}_{0,r}(III)$ is holomorphic in $\RE(s)>\frac{3}{4}$.
\end{lemma}

\begin{proof}
 Let $c > 1$.  This contribution may be written
 \begin{align*}
 &\frac{4}{3\xi(1+z)} \int_0^1 \frac{d\lambda}{\lambda} \lambda^{8s -6} \int_0^\infty \frac{dt}{t}t^{4s-1} \int_{-\infty}^\infty du \oint_{\RE \alpha = c}d\alpha \frac{\psi(\alpha)t^{\alpha}}{\alpha(w-\alpha)}  \\&\times \sum_{m=1}^\infty \eta_{\frac{z}{2}}(m) K_{\frac{z}{2}}(2\pi mt^2) \cos(2\pi m u)\sum_{\ell \neq 0} \sum_{b \neq 0} \sum_{\substack{0 \leq d < b\\ (b,d)=1}} \hat{f}_0\left(\lambda^{-2} \ell\left(tb x + \frac{1}{t}(bu+d)y\right)^2\right).
 \end{align*}
 Perform M\"{o}bius inversion to eliminate the condition $(d,b)=1$.  This obtains
 \begin{align*}
  &= \frac{4}{3\xi(1+z)}\int_0^1 \frac{d\lambda}{\lambda} \lambda^{8s -6} \int_0^\infty \frac{dt}{t}t^{4s-1}\int_{-\infty}^\infty du \oint_{\RE(\alpha)=c}d\alpha \psi(\alpha) \frac{t^{\alpha}}{\alpha(w-\alpha)} \sum_{b_1,b_2} \frac{\mu(b_1)}{b_1} \\ &\times \sum_{m=1}^\infty \eta_{\frac{z}{2}}(b_2m)K_{\frac{z}{2}}(2\pi m b_2 t^2) \cos\left(2\pi \frac{m}{b_1}u\right)
  \sum_{\ell \neq 0} \hat{f}_0\left(\lambda^{-2}  \ell\left(t b_1b_2 x + \frac{uy}{t}\right)^2\right).
 \end{align*}
 Make a change of variable in $u$, then in $t$ to put the argument of $\hat{f}_0$ in standard form
 \begin{align*}
  &= \frac{4}{3\xi(1+z)}\int_0^1 \frac{d\lambda}{\lambda} \lambda^{8s - 6} \int_0^\infty \frac{dt}{t} \int_{-\infty}^\infty du \oint_{\RE(\alpha) = c} d\alpha \frac{\psi(\alpha)t^{4s+1+\alpha}}{\alpha(w-\alpha)} \sum_{b_1,b_2} \mu(b_1)b_2 
  \\ & \times \sum_{m=1}^\infty \eta_{\frac{z}{2}}(b_2m) K_{\frac{z}{2}}(2\pi m b_2 t^2) \cos(2\pi m b_2 t^2 u) \sum_{\ell \neq 0} \hat{f}_0(\lambda^{-2}t^2 b_1^2 b_2^2 \ell(x + uy)^2)\\
  &=\frac{4}{3\xi(1+z)}\int_0^1 \frac{d\lambda}{\lambda} \lambda^{12s -5 + \alpha} \int_0^\infty \frac{dt}{t} \int_{-\infty}^\infty du \oint_{\RE(\alpha) = c}d\alpha \frac{\psi(\alpha)t^{4s+1 + \alpha}}{\alpha (w-\alpha)}  \sum_{b_1,b_2} \frac{\mu(b_1)b_2}{(b_1b_2)^{4s+1 + \alpha} }\\
  &\times \sum_{\ell=1}^\infty \frac{1}{\ell^{2s + \frac{1+\alpha}{2}}}\sum_{m=1}^\infty \eta_{\frac{z}{2}}(b_2m) K_{\frac{z}{2}}\left(\frac{2\pi m \lambda^2 t^2}{\ell b_1^2 b_2} \right)\cos\left(\frac{2\pi m \lambda^2 t^2u}{\ell b_1^2 b_2} \right)\sum_{\epsilon = \pm}\hat{f}_0(\epsilon t^2(x+uy)^2).
 \end{align*}
Open the Bessel function and cosine with Mellin transform.  
\begin{align*}
 &\frac{4}{3\xi(1+z)} \oint_{\RE(\alpha) = c}d\alpha \frac{\psi(\alpha)}{\alpha(w-\alpha)}\oint_{\alpha_1, \alpha_2}d\alpha_1 d\alpha_2 \frac{\Gamma(\alpha_1)}{(2\pi)^{\alpha_1 +\alpha_2}} \cos\left(\frac{\pi}{2}\alpha_1\right)2^{\alpha_2-2}\Gamma\left(\frac{\alpha_2+\frac{z}{2}}{2} \right)\Gamma\left(\frac{\alpha_2 - \frac{z}{2}}{2} \right)\\
 &\times \int_0^1 \lambda^{12s -5 - 2(\alpha_1+\alpha_2) +\alpha} \sum_{\ell, m, b_1, b_2 = 1}^\infty \frac{\mu(b_1) \eta_{\frac{z}{2}}(b_2m)}{m^{\alpha_1 + \alpha_2} \ell^{2s + \frac{1+\alpha}{2}-\alpha_1-\alpha_2} b_1^{4s+1+\alpha-2(\alpha_1+\alpha_2)}b_2^{4s+\alpha-(\alpha_1+\alpha_2)}}\\
 &\times\int_0^\infty \frac{dt}{t} t^{4s+1 + \alpha -2(\alpha_1 + \alpha_2)} \int_{-\infty}^\infty \frac{du}{u} |u|^{-\alpha_2 + 1} \sum_{\epsilon = \pm}\hat{f}_0\left(\epsilon t^2\left(x + uy\right)^2\right)
\end{align*}
Meromorphically continue the integrand in $\alpha$ to $\alpha = 0$ to pick up the pole. In order to gain absolute convergence in the Dirichlet series one requires $\RE(\alpha_1 + \alpha_2)>1$, $\RE(2s +\frac{1}{2}- \alpha_1 - \alpha_2) > 1, \RE(4s+1 -2(\alpha_1+\alpha_2))>1, \RE(4s-(\alpha_1 + \alpha_2))>1$, all of which hold with $\RE(\alpha_1 + \alpha_2) = 1 + \varepsilon$ and $\RE(s) > \frac{3}{4}$ if $\varepsilon > 0$ is sufficently small.    This obtains the claim.
\end{proof}

\subsection{Square discriminants}
\begin{lemma}
 The sum over square discriminants,
 \[
  Z_{\square}^{\pm, r} = \frac{2}{9} \int_0^\infty \frac{d\lambda}{\lambda} \int_0^\infty \frac{dt}{t^3} \int_0^1 du \lambda^{8s} t^{4s} \E_r(n_u^{t}a_t)\sum_{\substack{x = (b,c,d)\\ \Disc(x) = \square}} f_{\pm,0}((d_\lambda a_t n_u)_2 \cdot x)
 \]
is holomorphic in $\RE(s) > \frac{3}{4}$.
\end{lemma}
\begin{proof}
 Split $\E_r$ into constant and non-constant terms, expressing the sum as 
 \[
  Z_{\square}^{\pm, r} = Z_{\square, c}^{\pm, r} + Z_{\square, n}^{\pm, r}.
 \]
We have, see Lemma \ref{square_discriminant_lemma},
\[
Z_{\square,c}^{\pm, r} =\frac{2}{9} Z_{\square}\left(f, 2s + \frac{1}{2} + \frac{z}{2}, s-\frac{1}{4} - \frac{z}{4}\right) +\frac{2}{9} \frac{\xi(z)}{\xi(1+z)}Z_{\square}\left(f, 2s + \frac{1}{2} -\frac{z}{2}, s - \frac{1}{4} +\frac{z}{4}\right),
\]
where 
\[
 Z_{\square}(f, s_1, s_2) = \frac{1}{4}\Xi(s_1, s_2) \Phi_+(f, s_1-1, s_2-1).
\]
By restricting the support of $f$ away from the singular set, assume $\Phi_+(f, s_1-1, s_2-1)$ is entire.  We have
\[
 \Xi(s_1, s_2) = \frac{\zeta(s_1)^2 \zeta(2s_1+2s_2-1)\zeta(2s_2)}{\zeta(2s_1)\zeta(s_1 + 2s_2)}
\]
is holomorphic in $\RE(s)> \frac{3}{4}$.

For the non-constant terms write
\begin{align*}
 Z_{\square, n}^{\pm, r} &= \frac{2}{9} \int_0^\infty \frac{d\lambda}{\lambda} \int_0^\infty \frac{dt}{t^3}\int_0^1 du \lambda^{8s}t^{4s} \E_r^n(n_u^{t}a_t) \sum_{\substack{x=(b,c,d)\\ \Disc(x) = D^2}} f_0((d_\lambda a_t n_u)_2 \cdot x)\\
 &= \frac{2}{9} \int_0^\infty \frac{d\lambda}{\lambda} \int_0^\infty \frac{dt}{t} \int_{-\infty}^\infty du \lambda^{8s}t^{4s-2}\\&\times \sum_{\substack{x = (b,c,d)\\ -|b|\leq c < |b|\\ \Disc(x) = D^2}}  \E_r^n\left(n_{u-\frac{c}{2b}}^{t}a_t\right) f_0\left(\lambda^2\left(t^2 b, 2bu, \frac{\frac{-D^2}{4} + b^2u^2}{t^2b}\right)\right)
 \\&= \frac{2}{9} \int_0^\infty \frac{d\lambda}{\lambda} \int_0^\infty \frac{dt}{t} \int_{-\infty}^\infty du \lambda^{8s}t^{4s-2}\\&\times \sum_{\substack{x = (b,c,d)\\ -|b|\leq c < |b|\\ \Disc(x) = D^2}}\frac{1}{bD^{4s-1}}  \E_r^n\left(n_{\frac{Du}{b}-\frac{c}{2b}}^{t}a_t\right) f_0\left(\lambda^2\left(\frac{t^2 b}{D}, 2u, \frac{D(\frac{-1}{4} + u^2)}{t^2b}\right)\right)\\
 &= \frac{2}{9} \int_0^\infty \frac{d\lambda}{\lambda} \int_0^\infty \frac{dt}{t} \int_{-\infty}^\infty du \lambda^{8s}t^{4s-2}\\&\times\sum_{\substack{x = (b,c,d)\\ -|b|\leq c < |b|\\ \Disc(x) = D^2}}\frac{1}{b^{2s}D^{2s}}  \E_r^n\left(n_{\frac{Du}{b}-\frac{c}{2b}}^ta_{t\sqrt{\frac{D}{b}}}\right) f_0\left(\lambda^2\left(t^2, 2u, \frac{(\frac{-1}{4} + u^2)}{t^2}\right)\right)
 \end{align*}
Use the bound $E_r^n(n_u a_t) \ll \frac{1}{t}$.  As in \cite{S75}, for the sum over $x$ to converge we require $\RE(2s - \frac{1}{2}) > 1$ or $\RE(s) > \frac{3}{4}$.  Similarly, for the integral over $t$ to converge, we need $\RE(4s-3)>0$ or $\RE(s)>\frac{3}{4}$.  This guarantees the absolute convergence.

\end{proof}

 Combining the above results proves that the reducible orbital zeta function is meromorphic in $\RE(s)> \frac{3}{4}$, with simple poles at $\frac{5\pm z}{4}$.  Matching these poles against the poles of the whole zeta function proves Theorem \ref{irreducible_theorem}.
\bibliographystyle{plain}

\end{document}